\documentclass{article}
\usepackage{amsmath,amssymb,stmaryrd,cite}
\usepackage[pdftex]{graphicx}
\usepackage{fullpage}
\usepackage{color}
\usepackage{cite}

\usepackage{amsmath,amsthm,amssymb}

\newcommand{\R}{\mathbb{R}}
\newcommand{\Spec}{{\mathrm{Spec}}}
\newcommand{\diag}{{\mathrm{diag}}}
\newcommand{\dist}{{\mathrm{dist}}}
\newcommand{\diam}{{\mathrm{diam}}}
\renewcommand{\phi}{\varphi}
\newcommand{\e}{\epsilon}
\newcommand{\av}[1]{\left|{#1}\right|}
\newcommand{\norm}[1]{\left\|{#1}\right\|}
\renewcommand{\P}{\mathbb{P}}
\newcommand{\cpg}{\oblong}
\newcommand{\mc}[1]{\mathcal{#1}}
\numberwithin{equation}{section}

\newtheorem{thm}{Theorem}[section]

\newtheorem{prop}[thm]{Proposition}
\newtheorem{corr}[thm]{Corollary}
\newtheorem{define}[thm]{Definition}
\newtheorem{example}[thm]{Example}
\newtheorem{remark}[thm]{Remark}
\newtheorem{obs}[thm]{Observation}
\newtheorem{conjecture}[thm]{Conjecture}

\newcommand{\be}{\begin{enumerate}}
\newcommand{\bi}{\begin{itemize}}
\newcommand{\ee}{\end{enumerate}}
\newcommand{\ei}{\end{itemize}}
\newcommand{\ii}{\item}

\newcommand{\cb}[1]{{{#1}}}

\title{The generalized distance spectrum of a graph and applications}
\author{Lee DeVille\\Department of Mathematics, University of Illinois}

\begin{document}

\maketitle

{\abstract{The generalized distance matrix of a graph is the matrix whose entries depend only on the pairwise distances between vertices, and the generalized distance spectrum is the set of eigenvalues of this matrix.  This framework generalizes many of the commonly studied spectra of graphs. We show that for a large class of graphs these eigenvalues can be computed explicitly. We also present the applications of our results to competition models in ecology and rapidly mixing Markov chains.}}



\section{Introduction}

\begin{define}\label{def:gdm}
Let $G = (V,E)$ be a graph with diameter $d$.  We define the {\bf generalized distance matrix} of $G$ as the matrix $
  \mathcal{M}(f;G)$, whose entries are given by 
\begin{equation*}
  \mathcal{M}(f;G)_{x,y} = f(\dist(x,y)),
\end{equation*}
where $\dist(x,y)$ is the length of the shortest path in $G$ connecting $x$ and $y$.   The {\bf generalized distance spectrum} is the spectrum of this matrix, which we denote $\Lambda(f;G) = \{\lambda_j(f;G)\}_j$ where $j$ ranges over all of the eigenvalues.  \end{define}

The matrix $\mathcal{M}(f;G)$ depends on a function $f$, but since $G$ has finite diameter, this really depends only on the $d+1$ numbers $f_m := f(m),$ for $m=0,\dots,d$.
\cb{While this definition is new, it} subsumes many of the matrices commonly associated to graphs:  
\be

\ii Choosing $f_m = \delta_{m,1}$ gives the adjacency matrix~\cite{Brouwer.Haemers.book};  

\ii More generally, choosing $f_m = \delta_{m,q}$ gives the distance-$q$ adjacency matrix;

\ii Choosing $f_m = m$ gives the distance matrix~\cite{Edelberg.Garey.Graham.76, Graham.Lovasz.78, Aouchiche.Hansen.14, Atik.Panigrahi.15}, here we will call this the {\bf classical distance matrix} to distinguish from our generalization;

\ii  If $G$ is $k$-regular with $f_0 = k$, $f_1 = -1$, and $f_i=0$ for all $i =2,3,\dots, d$, we obtain the graph Laplacian~\cite[Chapter 1]{Fan.Chung.book}.

\ee

\begin{define}
Let $G=(V,E)$ be a graph.  For each $x\in V(G)$ we define $G_i(x)$ as the set of vertices that are distance $i$ from $x$.  For any $x,y\in V(G)$ and $j,k=0,\dots,d$ we define
\begin{equation*}
  n_{j,k}(x,y) = \av{G_j(x) \cap G_k(y)}.
\end{equation*}
Following~\cite{Brouwer.Cohen.Neumaier.book, Koolen.Shpectorov.94, vanDam.Haemers.02, vanDam.Koolen.Tanaka.16}, we say that $G$ is {\bf distance-regular} if this number is a function only of the distance between $x,y$, i.e.
\begin{equation*}
 n_{j,k}(x,y) = n_{j,k}(d(x,y)).
\end{equation*}
Now choose any pair $x,y$ with $d(x,y) = i$, and define $n^i_{jk}$ as the size of this set.
\end{define}

One of the main results of this paper is that \cb{(Theorem~\ref{thm:main})} if $G$ is distance-regular, or \cb{(Theorem~\ref{thm:sum})} if $G$ is  a Cartesian product of distance-regular graphs, then the eigenvalues of $\mc{M}(f;G)$ are linear in the components $f_i$ and there is an algorithm for computing the coefficients of the linear expression.  \cb{(This algorithm requires us to diagonalize a $(d+1)\times(d+1)$ matrix but this single diagonalization is enough to completely determine the spectrum.)}  Since in this case the spectrum is linear in the $f_m$, the choice $f_m = z^m$ gives the generating functions of the formulas for general $f$, and this gives a compact representation of the spectrum.  \cb{We also present conditions guaranteeing positive-definiteness of the matrix $\mc{M}(f;G)$ in Section~\ref{sec:positivity} and relate these conditions to embeddability properties of the graph in Section~\ref{sec:hierarchy}.}

There are many infinite families of distance-regular graphs known (see~\cite{Brouwer.Cohen.Neumaier.book, vanDam.Koolen.Tanaka.16} for many examples) and their adjacency spectrum, Laplacian spectrum, distance spectrum, etc. have been studied in great detail.   \cb{Our proof is a generalization of both the classical techniques for distance-regular graphs~\cite{Godsil.Royle.book, Brouwer.Haemers.book, Brouwer.Cohen.Neumaier.book} and on more recent techniques for the classical distance spectrum~\cite{Atik.Panigrahi.15}.}

\section{Main Results}\label{sec:theory}

Much of the theoretical background below is known classically and traditionally focuses on the adjacency spectrum~\cite{Godsil.Royle.book, Brouwer.Haemers.book, Brouwer.Cohen.Neumaier.book}.  Recently, this approach was generalized to compute many properties of the distance spectrum in~\cite{Atik.Panigrahi.15}. Our method here is a generalization of \cb{the approach of~\cite{Atik.Panigrahi.15}} but we give all of the details here for completeness.  

\subsection{Distance-regular Graphs}

We first state the main result:

\begin{thm}\label{thm:main}
  Let $G$ be a distance-regular graph with diameter $d$(so that $\mc{M}(f;G)$ has entries that depend on the $d+1$ quantities $f_0,\dots,f_d$). Then $\Lambda(f;G)$ consists of $d+1$ distinct linear functions of $f_m$, and the multiplicities of these eigenvalues are the same as for the adjacency spectrum.  Moreover, there is a list of $d+1$ matrices $Q_m, m=0,\dots, d$ of size $(d+1)\times (d+1)$ such that if the eigenvalues of $Q_m$ are $\lambda_{i,m}$, i.e. $Q_m v_i = \lambda_{i,m}v_i$, then the eigenvalues of $\mc{M}(f;G)$ are 
  \begin{equation}\label{eq:defoflambda}
  \lambda_i(f;G) := \sum_{m=0}^d f_m \lambda_{i,m}.
  \end{equation}
  Finally, the matrix $Q_m$ can be written as $Q_m = p_m(Q)$, where $p_m(\cdot)$ is an explicitly-determined polynomial of degree {$m$}, and $Q$ can be explicitly determined.
\end{thm}

We first make some observations from linear algebra.  Since $G$ is a regular graph, we have $A{\bf 1} = \cb{\deg(G)}{\bf 1}$ where $\cb{\deg(G)}$ is the valency of $G$.  We will use the convention throughout that $\lambda_0 =  \cb{\deg(G)}$ and $v_0 = {\bf 1}$.  Since $A$ is symmetric, this implies that the other eigenvalues are real and the other eigenvectors are orthogonal to ${\bf 1}$.

It can be shown that if we know $n^k_{j,1}$ for $j=k\pm 1$, then we know all of the remaining intersection numbers.  More concretely, let us define the numbers $b_0,\dots, b_{d-1}$ and $c_1,\dots, c_d$ as follows:  if $d(x,y) = k$, then 
\begin{equation*}
  b_k = \av{G_{k+1}(x)\cap G_1(y)},\quad c_k = \av{G_{k-1}(x)\cap G_1(y)}.
\end{equation*}
The sequence $\{b_0,\dots,b_{d-1};c_1,c_2,\dots,c_d\}$ is called the {\bf intersection array} of the graph.  Let us define $a_k = \deg(G) - b_k-c_k$.  Now let $A = A_1$ be the adjacency matrix of $G$, and define $A_k$ as the distance-$k$ matrix.  \cb{That is to say, $A_k$ is a zero-one matrix where the one entries correspond to vertices of distance exactly $k$.}  Then we have the recurrence
\begin{equation}\label{eq:rec}
  AA_i = c_{i+1} A_{i+1} + a_i A_i + b_{i-1}A_{i-1}.
\end{equation}
It follows directly from this recurrence that $A_k$ can be written as $A_k = p_k(A)$ where $p_k$ is some polynomial of degree $k$, and that there is a $d+1$-degree polynomial $p_{d+1}$ such that $p_{d+1}(A) = 0$.  From this it follows that $A$ has exactly $d+1$ distinct eigenvalues:  since the $A_k$ are linearly independent, there are at least $d+1$ distinct eigenvalues, but since $p_{d+1}(A) = 0$, there are at most $d+1$ --- and in fact, they are the roots of $p_{d+1}(z) = 0$.  This approach is laid out in~\cite{Godsil.Royle.book, Brouwer.Haemers.book, Brouwer.Cohen.Neumaier.book} and has been used to analyze the adjacency spectrum of many distance-regular graphs.  In fact, much more is known here: the algebraic structure described above shows that these matrices form an association scheme; this and other deep theory allow for  strong results on the classification of distance-regular graphs, but we do not use this here.

More recently, an extension of these ideas to compute the (classical) distance spectrum was laid out in~\cite{Atik.Panigrahi.15}; recall here that this is our framework with the choice of $f_i = i$ for all $i=0,\dots,d$.  We modify the approach of~\cite{Atik.Panigrahi.15} for general $f$ and gives us the proof of Theorem~\ref{thm:main}.

{\bf Proof of Theorem~\ref{thm:main}:}   Let us form the $(d+1)\times (d+1)$ tridiagonal matrix $Q$ by defining the superdiagonal to be the vector $b$, the subdiagonal to be the vector $c$, and we choose the diagonal elements so that each row has the row sum equal to the degree of a vertex.  More specifically, we have 
\begin{equation*}
  Q_{k,k+1} = b_k, \quad Q_{k,k-1} = c_k, \quad Q_{k,k} = a_k := \deg(G) - b_k - c_k.
\end{equation*}

Choose any vertex $x\in V(G)$, and consider the sets $G_k(x)$ for $k=1,\dots, d$.  We order $V(G)$ as (the flattened version of) $\{x,G_1(x),G_2(x),\dots, G_d(x)\}$. Consider the matrix $\mc{M}(f;G)$, which we break up into blocks by defining $N_{jk}$ as the $\av{G_j(x)}\times \av{G_k(x)}$ matrix with rows from $G_j(x)$ and columns from $G_k(x)$.  It follows from the definition of graph-regular that each of the $N_{j,k}$ has constant row sum, and in fact this row sum is $\sum_{m=0}^d n^j_{km}f_m$.  To see this, fix $y\in G_j(x)$ and vary $z\in G_k(x)$ (this corresponds to one row of $N_{jk}$).  If we ask how many of these $z$ are distance $m$ from $y$, we are asking for the size of $G_k(x)\cap G_m(y)$, and since $d(x,y) = j$ this is $n^j_{km}$. Each of the terms corresponding to distance $m$ is $f_m$ and so the sum is as above. 

Also, we have that $A$ and $Q$ are isospectral.  Note that since each of the blocks have constant row sum, if $Qv = \lambda v$ with $v = (v^{(0)},v^{(1)},\dots,v^{(d)})$, then { if  for each $m=0,\dots,d$, let $w^{(m)} = v^{(m)}\otimes 1_{G_m(x)}$, and then $w = (w^{(0)},\dots w^{(d)})$ is also an eigenvalue of $A$ with eigenvalue $\lambda$}.

Now, let us replace each of these blocks by their row sum, i.e. consider the $(d+1)\times(d+1)$ matrix $\mc{Q}(f;G)$ defined by 
\begin{equation*}
  \mc{Q}(f;G)_{jk} = \sum_{m=0}^d n^j_{km}f_m.
\end{equation*}
Of course, we can write $\mc{Q}(f;G) = \sum_{m=0}^d f_m Q_m$ where the $Q_m$ do not depend on $f$.  But just as $Q_1$ was the block-average of $A_1 = A$, and we can see that $Q_m$ is the block-average of $A_m$ (in fact, we obtain $A_m$ by choosing $f_i = \delta_{i,m}$ and the earlier argument applies again.)  From this it follows that the $Q_m$ satisfy the same recurrence as~\eqref{eq:rec}:
 $Q_0 = I, Q_1=Q$, and 
\begin{equation*}
  QQ_m = c_{m+1} Q_{m+1} + a_m Q_m + b_{m-1}Q_{m-1}.
\end{equation*}
It follows directly from this recurrence that $Q_k$ can be written as $Q_k = p_k(Q)$ where $p_k$ is the same polynomial of degree $k$ as above.  Moreover, since each $Q_m$ is a polynomial function of $Q$, any pair $Q_m,Q_{m'}$ commute.  This means that the eigenvalues of $\mc{Q}(f;G)$ are linear in the $f_i$:  if we have $Q_1 v = {\lambda} v$, then $Q_m v = p_m(\lambda) v$ as well.  So let us denote the $d+1$ eigenvectors of $Q$ by $v_i$ and the associated eigenvalues for $Q_m$ by $\lambda_{i,m}$, then
\begin{equation*}
  \mc{Q}(f;G) v_i = \sum_{m=0}^d f_mQ_m v_i = \sum_{m=0}^d f_m\lambda_{i,m}v_i.
\end{equation*}

\qed

\begin{remark}
  As noted in the theorem, the multiplicities of $\lambda_i(f;G)$ are, for generic $f$, the same as they are for the adjacency matrix.  In most cases below we will not belabor the point as we are interested in obtaining the formulas for the eigenvalues; for example, since the multiplicities of the eigenvalues are the same as for the adjacency matrix, the standard theory~\cite[Chapter 12]{Brouwer.Haemers.book} for determining their multiplicities applies.  In particular in the examples in Section~\ref{sec:examples} we will usually discuss only the eigenvalues themselves, with a few exceptions.
\end{remark}

We can now a compact description of the spectrum of $\mc{M}(f;G)$.  We first form the matrix $Q$ and compute its spectrum, writing $Q v_i = \lambda_i v_i$.  We then have
\begin{equation*}
  Q_m = p_m(Q), \quad Q_m v_i = p_m(\lambda_i) v_i,
\end{equation*}
and therefore
\begin{equation*}
  \mc{Q}(f;G) v_i = \sum_{m=0}^d f_m p_m(\lambda_i)v_i,
\end{equation*}
which gives the compact formulas
\begin{equation}\label{eq:lambdaf}
  \lambda_{m,i} = p_m(\lambda_i),\quad \lambda_i(f;G) = \sum_{m=0}^d f_m p_m(\lambda_i).
\end{equation}
Thus, in theory all we need to know are the eigenvalues of $Q$ and the recurrence relation~\eqref{eq:rec} and we have everything.  

\begin{define}
  Let us denote $\lambda_i$, $i=0,\dots, d$ as the eigenvalues of {$Q$}, and let $p_m(z)$ be the polynomials defined as in the recurrence relation~\eqref{eq:rec}, i.e.
  \begin{equation*}
    p_0(x) = 1,\quad p_1(x) = x,\quad p_1(x)p_m(x) = c_{m+1}p_{m+1}(x) + a_mp_m(x) + b_{m-1}p_{m-1}(x).
  \end{equation*}
  Let us also define
\begin{equation}\label{eq:defofphi}
   \phi_{i}(z;G) = \sum_{m=0}^d z^m \lambda_{i,m} = \sum_{m=0}^d z^m p_m(\lambda_i),
\end{equation}  
and
\begin{equation*}
  q_m(x) = \sum_{l=0}^m p_l(x).
\end{equation*}
\end{define}

\begin{prop}\label{prop:factor}
  For any $i=1,\dots, d$ (note: not $i=0$) we have
  \begin{equation*}
  \phi_i(z) = (1-z)\cdot\sum_{m=0}^{d-1} \left(q_m(\lambda_i)z^m\right),\quad p_d(\lambda_i) = -q_{d-1}(\lambda_i),\quad  q_d(\lambda_i) = 0.
\end{equation*}
\end{prop}

\begin{proof}
This follows from the observation that if we choose $f_i = 1$ for all $i$, then $\mc{M}({\bf 1};G)$ is the all-ones matrix; the spectrum of which is a single eigenvalue of $n$ and $n-1$ eigenvalues of zero.  In particular, this implies from~\eqref{eq:lambdaf} that for $i=1,\dots, d,$
\begin{equation}
  \sum_{m=0}^d p_m(\lambda_i) = 0.
\end{equation}
and thus $p_d(\lambda_i) = -q_{d-1}(\lambda_i)$ and $q_d(\lambda_i) = 0$ for $i=1,\dots, d$.   Moreover, since $\lambda_i$ are the roots of $q_d(\cdot)$ for $i=1,\dots,d$, and $\lambda_0 = k$, we have a quick factorization of the characteristic polynomial:
\begin{equation}
  \det(Q - \lambda I) = \pm (\lambda-k)q_d(\lambda).
\end{equation}
Again note that $\phi_i(1) = 0$, so that it has a factor of $1-z$.  But also noting that $q_m(z) -q_{m-1}(z) = p_m(z)$,
we have the factorization
\begin{equation*}
  \phi_i(z) = (1-z)\cdot\sum_{m=0}^{d-1} \left(q_m(\lambda_i)z^m\right).
\end{equation*}
\end{proof}

\begin{prop}\label{prop:0}
  If $G$ is distance-regular, then $\phi_0(z) = \sum_{m=0}^d n_m z^m$ where $n_m = \av{G_m(x)}$ is the number of vertices at distance $m$ from any given vertex.  In particular, the coefficients of $\phi_0(z)$ are positive.
\end{prop}

\begin{proof}
  Note that $\lambda_0 = \cb{\deg(G)}$ since $Q{\bf 1} = \cb{\deg(G)}{\bf 1}$.  Similarly, it is not hard to see that the row sum of any $Q_m$ is $n_m$, and therefore $\lambda_{m,0} = n_m$, and the result follows.
\end{proof}

\begin{define}\label{def:z}
  We call the individual $\phi_i(z;G)$ the {\bf spectral polynomials} of $G$.  For compactness, we will write them as a set, or sometimes a multiset as the vector $\Phi(z;G)$.  { We will also abuse notation slightly and denote $\mc{M}(z;G)$ as the generalized distance matrix where we have chosen $f_k = z^k$.} 
\end{define}

\begin{remark}\label{rem:derivative}
 Note that computing $\Phi(z;G)$ basically determines all of the common graph invariants.  For example, the eigenvalues of the adjacency matrix can be computed as $d/dz (\Phi(z;G))$ evaluated at $z=0$, and the eigenvalues of the classical distance matrix are $d/dz (\Phi(z;G))$ evaluated at $z=1$.

 One common observation in the literature on distance matrices~\cite{Ruzieh.Powers.90, Indulal.Gutman.08, Lin.etal.13, Aouchiche.Hansen.14, Azarija.14, Barik.Bapat.Pati.15, Aalipour.etal.16} is that they can have eigenvalues that occur in different multiplicities than for the adjacency matrix.  \cb{It has been observed that many examples of distance-regular graphs have fewer distinct distance eigenvalues than adjacency eigenvalues.  The reason for this in the case of distance-regular graphs} is clear once we consider the properties of the spectral polynomials $\phi_i(z;G)$; while the functions $\phi_i(z;G)$ are all distinct, in many cases they have common derivatives at $z=1$.  In fact, we see below that for some Hamming and Johnson graphs, the functions $\phi_i(z;G)$ are typically have multiple factors of $(1-z)$.  This implies that the derivative at $z=1$ has multiple zeros, and thus the multiplicity of the zero eigenvalue is much higher for the classical distance matrix.  In fact, it follows from above that when the { multiplities of the} eigenvalues of any graph matrix are different from those of the adjacency matrix, it is nongeneric and due to a coincidental arrangement of these $\phi_i(z;G)$ at a particular value of $z$.  
 
 In fact, it follows from the above that if we consider a generic perturbation of the classical distance matrix (e.g. instead of $f_i = i$ we choose $f_i = i + \e X_i$ for some independent random $X_i$, then with probability one the spectrum will have the exact same multiplicities as for the adjacency spectrum and we will have ``unfolded'' the coincidence that occurs in the derivatives at $z=1$.  In this sense, the eigenvalue multiplicities are more stable for the adjacency matrix than they are for the distance matrix.
\end{remark}

\subsection{Products of graphs}

Here we present some results for the generalized distance matrices of direct sums of graphs.  The main result of this section is that the eigenvalues of the generalized distance matrix of a Cartesian sum of graphs can be written as a tensor product of matrices on the individual graphs, and this implies a multiplicativity property of eigenvalues.

\begin{define}
  Let $G,H$ be graphs with diameters $d_G, d_H$.  We define the {\bf Cartesian product} of $G$ and $H$, denoted $G\cpg H$ as the graph with vertex set $V(G)\times V(H)$ and we say $(x_1,x_2)$ is adjacent to $(y_1,y_2)$ if $x_1=y_1$ and $x_2$ is adjacent to $y_2$ in $H$, or if $x_2=y_2$ and $x_1$ is adjacent to $y_1$ in $G$.
\end{define}

\begin{prop}\label{prop:additive}
If $x,y\in V(G\cpg H)$, where $x=(x_1,x_2), y=(y_1,y_2)$ then 
\begin{equation*}
  \dist_{G\cpg H}(x,y) = \dist_G(x_1,y_1) + \dist_H(x_2,y_2).
\end{equation*}
\end{prop}

\begin{proof}
  This is straightforward, but see~\cite[Lemma 1]{Stevanovic.04} for this and related formulas.
\end{proof}

\begin{remark}
  Note that it follows from Proposition~\ref{prop:additive} that if $G$ has diameter $d_G$ and $H$ has diameter $d_H$, then $G\cpg H$ has diameter $d_G+d_H$.  We remark here that the direct sum of distance-regular graphs is not necessarily distance-regular~\cite{Song.86, Aggarwal.Jha.Vikram.00, Stevanovic.04, Stevanovic.Indulal.09}, but we will still be able to analyze the generalized distance spectrum of these sums.  
\end{remark}

\begin{define}
  Let $M$ be an $m\times m$ matrix and $N$ be an $n\times n$ matrix.  We define the {\bf Kronecker product} (or {\bf tensor product}) of $M$ and $N$, denoted $M\otimes N$, as the $(mn)\times(mn)$ matrix whose elements are defined as follows.  Let $a,c\in[m]$ and $b,d\in[n]$, and then
  \begin{equation*}
    (M\otimes N)_{(a,b),(c,d)} = M_{a,c}\cdot N_{b,d}.
  \end{equation*}
Equivalently, if $v\in \R^m$ and $w\in \R^n$, let us define $v\otimes w$ as the vector in $\R^{mn}$ whose entries are given by $v_iw_j$, where we sum over the indices lexicographically, and then $M\otimes N$ is the linear map on $\R^{mn}$ such that
\begin{equation*}
  (M\otimes N)(v\otimes w) = (Mv)\otimes (Nw).
\end{equation*}
From this it follows more generally for matrices that
\begin{equation*}
  (A\otimes B)\cdot (C\otimes D) = (AC)\otimes(BD).
\end{equation*}
\end{define}

\begin{remark}
  If $Mv = \mu v, Nw = \omega w$, then $(M\otimes N)(v\otimes w) = \mu\omega(v\otimes w)$.  Thus the eigenvalues of $M\otimes N$ are all possible products of eigenvalues of $M$ and eigenvalues of $N$, i.e. $\Spec(M\otimes N) = \Spec(M)\otimes \Spec(N)$.  Note also that this last formula works if we think of the eigenvalues as a set or as a multiset where we carry along multiplicities in the obvious fashion.  We will abuse notation by moving back and forth between the two conventions with abandon.
\end{remark}

\begin{prop}\label{prop:tensor}
If we let $A^{(G)}_k$ be the $k$th adjacency matrix of $G$ as defined  \cb{in the text preceding~\eqref{eq:rec}}, then
\begin{equation*}
  A^{(G\cpg H)}_k = \sum_{m=0}^k A^{(G)}_m \otimes A^{(H)}_{k-m}.
\end{equation*}
\end{prop}

\begin{proof}
  This is, in fact, just a fancy restatement of Proposition~\ref{prop:additive}.  To see this, consider $x,y\in V(G\cpg H)$.  Writing $x=(x_1,x_2), y=(y_1,y_2)$, we have $(A^{(G\cpg H)}_k)_{x,y} = 1$ iff $d_{G\cpg H}(x,y) = k$.  This is true iff there is a unique $m\in\{0,\dots, k\}$ with $d_G(x_1,y_1) =m$ and $d_G(x_2,y_2)=k-m$.  This means that $(A^{(G)}_m)_{x_1,y_1} =1$ and $(A^{(H)}_{k-m})_{x_2,y_2} = 1$, so that 
 \begin{equation*}
   \left(A^{(G)}_m \otimes A^{(H)}_{k-m}\right)_{x,y}
 \end{equation*}
is equal to 1 for exactly one value of $m$, and thus the sum is 1. 
\end{proof}

\begin{thm}\label{thm:sum}
  Let $G, H$ be graphs with diameter $d_G, d_H$ respectively. Let $g=(g_0, g_1,\dots, g_{d_G})$ and $h = (h_0, h_1, \dots, h_{d_H})$ be two vectors with the property that $g_k h_l$ depends only on $k+l$.   Define $f_{k+l}$ to be this common value, and note that $f = (f_0,f_1,\dots, f_{d_G+d_H})$.  Then
  \begin{equation*}
    \mathcal{M}(f;G\cpg H) = \mathcal{M}(g;G)\otimes \mathcal{M}(h;H).
  \end{equation*}
More generally, assume that $g^{(1)},\dots, g^{(p)}$ are $p$ vectors in $\R^{d_G+1}$ and $h^{(1)},\dots, h^{(p)}$ are $p$ vectors in $\R^{d_H+1}$ with the property that
\begin{equation*}
  \sum_{q=1}^p g^{(q)}_k h^{(q)}_l
\end{equation*}
depends only on $k+l$.  Again define $f_{k+l}$ to be this common value.  Then
\begin{equation*}
  \mathcal{M}(f;G\cpg H) = \sum_{q=1}^p \left(\mathcal{M}(g^{(q)};G)\otimes \mathcal{M}(h^{(q)};H)\right).
\end{equation*}
\end{thm}

\begin{proof}
  We first prove the result with $p=1$.  Note by definition that
\begin{equation*}
  \mathcal{M}(g;G) = \sum_{k=0}^{d_G} g_k A^{(G)}_k,\quad   \mathcal{M}(h;H) = \sum_{l=0}^{d_H} h_l A^{(H)}_l.
\end{equation*}
We then have
\begin{align*}
  \mathcal{M}(g;G)\otimes \mathcal{M}(h;H)
  	&= \left(\sum_{k=0}^{d_G} g_k A^{(G)}_k\right)\otimes \left(\sum_{k=0}^{d_G} h_l A^{(H)}_l\right)\\
	&= \sum_{k=0}^{d_G}\sum_{l=0}^{d_H}g_k h_l \left(A^{(G)}_k\otimes A^{(H)}_l\right)\\
	&= \sum_{k=0}^{d_G}\sum_{l=0}^{d_H} f_{k+l} \left(A^{(G)}_k\otimes A^{(H)}_l\right).
\end{align*}
Writing $n=k+l$, or $k=n-l$, this is the same as 
\begin{equation*}
  \sum_{n=0}^{d_G+d_H} f_n \sum_{l=0}^n\left(A^{(G)}_{n-l}\otimes A^{(H)}_l\right) = \sum_{n=0}^{d_G+d_H} f_n A^{(G\cpg H)}_n = \mc{M}(f;G\cpg H),
\end{equation*}
and we are done.  The proof for general $p$ is quite similar: start with a sum over $q$ on the outside, then pull it inside to form $f$, and this is otherwise the same.
\end{proof}

\begin{remark}
This theorem includes several special cases already known in the literature.  For example, it is well known~\cite[Section 1.4.6]{Brouwer.Haemers.book} that 
\begin{equation*}
  A^{(G\cpg H)} = A^{(G)}\otimes I_{\av{V(H)}} + I_{\av{V(G)}}\otimes A^{(H)},
\end{equation*}
and this corresponds to taking the vectors $g^{(1)} = h^{(2)} = (0,1,0,0,0,\dots)$ and $g^{(2)} = h^{(1)} = (1,0,0,0,\dots)$.

Similarly, let us choose $g^{(1)} = (1,2,3,\dots,d_G)$, $h^{(2)} = (1,2,3,\dots, d_H)$, and $g^{(2)} = h^{(1)} = {\bf 1}$.  Writing $D(G)$ as the classical distance matrix, this recovers the formula~\cite[Theorem 2.1]{Indulal.09}:
\begin{equation*}
  D(G\cpg H) = D(G)\otimes J_{\av{V(H)}} + J_{\av{V(G)}}\otimes D(H)
\end{equation*}

We can recover any $f$ we like by taking $p$ large enough.  For example, the formula of {Theorem~\ref{thm:sum}} can be recovered by choosing the $k+1$ pairs $g^{(q)} = e_q, h^{(q)} = e_{k-q}$, and then we can form any $f$ we would like through linearity (at the cost, perhaps, of having to choose $k+1$ vectors).
\end{remark}

\begin{corr}\label{corr:phi}
We have
\begin{equation*}
  \mathcal{M}(z;G\cpg H) = \mc{M}(z;G)\otimes \mc{M}(z;H)
\end{equation*}
and thus
  \begin{equation*}
    \Phi(z;G\cpg H) = \Phi(z;G)\otimes \Phi(z;H).
  \end{equation*}
\end{corr}

\begin{proof}
  Let us choose $g = (1,z,z^2,\dots, z^{d_G})$ and $h=(1,z,z^2,\dots,z^{d_H})$ in Theorem~\ref{thm:sum}, then we have $f = (1,z,z^2,\dots,z^{d_G+d_H})$ and the result follows. 
\end{proof}

Nothing in this section to this point has used the fact that the graphs are distance-regular, but now we can exploit these results to a general Cartesian product of (two or more) distance-regular graphs.

\begin{corr}\label{corr:tensor}
  Let $G_1,\dots, G_n$ be distance-regular graphs with $\diam(G_n) = d_n$.  Then $\Phi(z;\bigbox_{i=1}^n G_i)$ has at most $\prod_{i=1}^n d_i$ distinct components which are given by $\bigotimes_{i=1}^n\Phi(z;G_i)$.
\end{corr}

\begin{example}
  Consider the product $G = K_{n_1}\cpg K_{n_2}$, where $K_n$ denotes the complete graph on $n$ vertices.  We have
  \begin{equation*}
    \Phi(z;K_{n_1})=\{1-z,1+(n_1-1)z\},\quad    \Phi(z;K_{n_2})=\{1-z,1+(n_2-1)z\},
  \end{equation*}
  and thus
  \begin{equation*}
    \Phi(z;G) = \{(1-z)^2, (1-z)(1+(n_1-1)z), (1-z)(1+(n_2-1)z), (1+(n_2-1)z)(1+(n_2-1)z)\}. 
  \end{equation*}
  Notice that these are distinct iff $n_1\neq n_2$.  Also note that $G$ is distance-regular iff $n_1=n_2$, and $\diam(G) = 2$.  So, for example, seeing four distinct eigenvalues for a diameter two graph is a clear indication that it is not distance-regular.
  
  More generally, note that $G = \bigbox_{i=1}^q K_{n_i}$ will have at most $2^q$ distinct eigenvalues, and will have exactly this many if the $n_i$ are distinct.
  
\end{example}

\begin{remark}\label{rem:multiplicativity}
  As we have shown, when we take Cartesian products, the eigenvalues multiply, in the sense shown in Corollary~\ref{corr:tensor}.  This might seem strange at first, since the standard property is that eigenvalues are additive when we take these products (this is true, for example, for the adjacency eigenvalues or the classical distance eigenvalues, as is well known).  However, if we tie Corollary~\ref{corr:tensor} with Remark~\ref{rem:derivative}, note that the standard eigenvalues are given by the derivatives of our spectral polynomials, and thus multiplication of the polynomials corresponds to addition when they evaluated at a particular point.  See also~\cite{Cichacz.etal.16, Ito.Terwilliger.09, Blasiak.Flajolet.11}.

  The formulas above work out very well with Cartesian products, but what drives this is that the distance in a Cartesian product is linear in the sense of Proposition~\ref{prop:additive}.  For other graph products, the distance function is not linear (e.g. for the tensor product it is a maximum and not a sum) and thus it is unlikely such a nice formula as that in Theorem~\ref{thm:sum} would exist.
\end{remark}

\cb{\subsection{Linearity of eigenvalues}

One of the conclusions of Theorem~\ref{thm:main} is that when $G$ is distance-regular, the eigenvalues of $\mc{M}(f;G)$ are linear in the $f_i$.  It is natural to ask if there is a more general class of graphs for which this property of linearity holds, and it turns out that there is, as we describe below.  Recall that we define $A_k$ as the zero-one matrix where $(A_k)_{xy} = 1$ iff $d(x,y)=k$, the matrix $A_1$ is the standard adjacency matrix, and $A_0 = I$.  

\begin{thm}\label{thm:commute}
  If the matrices $A_k,A_l$ commute for all $k,l=0,\dots, d$, then the eigenvalues of $\mc{M}(f;G)$ are linear in the $f_i$. 
  \end{thm}
  
\begin{proof}
  Since the $A_k$ are real symmetric matrices, by~\cite[]{Horn.Johnson.book}, they commute iff they are simultaneously diagonalizable, which of course implies that they all share the same eigenvectors.  In particular,  there is a single matrix $P$ such that $A_k  =PD_kP^{-1}$ where $D_k$ is diagonal, for all $k$.  Then we have
  \begin{equation}
    \mc{M}(f;G) = \sum_{k=0}^d f_k A_k = \sum_{k=0}^d f_k PD_kP^{-1} = P\left(\sum_{k=0}^d f_k D_k\right)P^{-1}.
  \end{equation}
  Thus we have diagonalized $\mc{M}(f;G)$, and the eigenvalues of the inner diagonal matrix are linear in the $f_i$.
  
\end{proof}

\begin{remark}
  Note that if the only conclusion desired is the statement about linearity, then Theorem~\ref{thm:commute} implies Theorem~\ref{thm:main}, and has a much simpler proof.  But note that the conclusions are weaker; in Theorem~\ref{thm:main} we give a semi-explicit formula for computing the coefficients in the linear relations which requires no more than computing the spectrum of a $(d+1)\times(d+1)$ matrix, whereas in Theorem~\ref{thm:commute} there is no control over the entries of $D_k$.
\end{remark}

A natural question to ask is whether the condition that the $A_k$ commute gives a broader class of graphs than distance-regular, and the answer here is yes.  In fact, one result of~\cite{Weichsel.82} is that the set of {\bf distance-polynomial} graphs, i.e.  those where $A_k$ is a polynomial of $A_1$ for all $k$, is a strict superset of distance-regular graphs.  See also~\cite{reddy2011pattern} for a related algebraic perspective.

}

\cb{\subsection{Positivity of eigenvalues}\label{sec:positivity}

In this section we present a variety of results involving the positivity of eigenvalues under certain assumptions on the parameters.  First, some definitions:

\begin{define}\label{def:cd} 
We say that the graph $G$ is  {\bf uniformly positive definite} if $\Phi(z;G)\ge 0$ for all $z\in[0,1]$, and {\bf weakly positive definite} if $\Phi(z;G)\ge0$ for all $z\in[\delta,1]$ for some $0<\delta<1$.  We also refer to the set of parameters $1 = f_0 \ge f_1 \ge f_2\ge\dots \ge f_d \ge 0$ as the {\bf competition domain}.
\end{define}

\begin{remark}
  We will motivate the relevance of the competition domain to applications in theoretical biology in Section~\ref{sec:GLVC} below.  Note that if $\mc{M}(f;G)$ is positive definite over the entire competition domain, then this implies that it is uniformly positive definite; if $f_k = z^k$ and $z\in[0,1]$ then $f$ is in the competition domain.  Moreover, we connect the two definitions of positive definiteness to the distance hierarchy in Section~\ref{sec:hierarchy} below.
\end{remark}

}

\begin{prop}\label{prop:onepositive}
Let us denote $\chi_i = \phi_i'(1)$.  Then $G$ is weakly positive definite if $\chi_i<0$ for $i=1,\dots, d$, and if $G$ is uniformly positive definite, then $\chi_i\le 0$ for all $i=1,\dots,d$.  (Note in both cases that we are not considering $i=0$, since $\chi_0>0$ in general.)  Moreover, the condition $\chi_i\le 0$ for all $i=1,\dots,d$ is equivalent to the condition that the classical distance matrix of the graph has exactly one positive eigenvalue.  
   
\end{prop}

\begin{proof}
  We first consider $i>0$.  Using the formula in Proposition~\ref{prop:factor}, we have $\phi_i(1) = 0$ for all $i>0$.  Clearly, if $\phi_i'(1) >0$, then $\phi_i(z)$ is negative for $z\nearrow 1$ and this breaks uniformity.   
  Using~\eqref{eq:defofphi}, we have $\phi_i'(1) = \sum_{m=0}^d mz^{m-1}\lambda_{i,m}$,  and at $z=1$ this recovers the eigenvalues of the classical distance matrix.  Therefore for $i>0$, the distance eigenvalues are also non-positive.    Finally, we also note from Proposition~\ref{prop:0} that $\phi_0'(1) = \sum_{m=0}^d mn_m >0$, and the $0$th eigenvalue is simple.
\end{proof}  
  
\begin{remark}
Graphs for which the classical distance matrix has exactly one positive eigenvalue have been thoroughly studied, \cb{see~\cite{Graham.Pollak.71, Graovac.Jashari.Strunje.85, Balasubramanian.90, Merris.90, Koolen.Shpectorov.94, Balasubramanian.95, Xiao.Gutman.03, Bapat.Kirkland.Neumann.05, Bose.Nath.Paul.11, Aalipour.etal.16}.  The connection between this condition and the metric hierarchy is explored in depth in the text~\cite{Deza.Laurent.book}; see more on this connection in Section~\ref{sec:hierarchy} below.}
\end{remark}

\begin{prop}\label{prop:minmax}
  For fixed $\gamma_m$, both the minimum and the maximum of the function
  \begin{equation*}
    f_0 + \sum_{m=1}^d \gamma_m f_m
  \end{equation*}
  in the competition domain is attained at some $f$ of the form $(1,1,1,\dots,1,0,0,\dots,0)$, i.e. $f_k = 1(k < d')$ for some $d' \le d$.
\end{prop}

\begin{proof}
One can prove this using the standard optimization machinery but there is a more direct argument that gives insight.  Let us assume that $f$ is not of the form given above, which implies that there exists $a$ such that $f_k = 1$ for all $k=0,1,\dots, a$ and that $0<f_{a+1}<1$.  Let us further define $b$ as the maximal index such that $f_b = f_{a+1}$.  Now, the number $\gamma_{i,a+1}+\dots+\gamma_{i,b}$ is either positive or negative.  In this case, we can increase our function by sliding all of the $f_{a+1},\dots,f_b$ up (resp.~down) and therefore this vector is not extremal.
\end{proof}

\begin{corr}
  For $d'=0,\dots,d$ compute the numbers
  \begin{equation*}
    \lambda_i({\bf 1}(m \le d');G) = \sum_{m=0}^{d'} \lambda_{i,m}.
  \end{equation*}
  Then $\lambda_{i}(f;G)$ is nonnegative over the competition domain iff these numbers are all nonnegative.
\end{corr}

\begin{proof}
  The forward direction is clear.  For the backward direction, use~\eqref{eq:lambdaf} and Proposition~\ref{prop:minmax}.
\end{proof}

It is clear from this Corollary that one can  efficiently determine whether or not $\lambda_i(f;G)$ is nonnegative over the competition domain, and from this whether or not $C = \mc{M}(f;G)$ is positive definite over the competition domain.

\begin{prop}\label{prop:min}
  Assume that $f$ is in the competition domain, and let $r = \sum_{m=0}^d f_m n_m$ be the row sum of $\mc{M}(f;G)$.  Then $\lambda_i(f;G) \ge 2-r$, and in particular, is strictly greater than $-r$.
\end{prop}

\begin{proof}
    Without loss of generality, we can consider $i>0$ after Proposition~\ref{prop:0}.  
We know that $Q_m$ is a matrix with row sum equal to $n_m$ and positive entries.  This and the Perron--Frobenius theorem~\cite[Section 8.4]{Horn.Johnson.book} imply that $\lambda_{m,i} \ge - n_m$ for all $m,i$.  Thus we have
\begin{align*}
  \lambda_i(f;G)
  	&= \sum_{m=0}^d f_m \lambda_{i,m} = 1+ \sum_{m=1}^d f_m \lambda_{i,m}\ge\\
	&\ge 1 - \sum_{m=1}^d f_m n_m \ge 1 - \sum_{m=1}^d n_m = -(r - n_0 - 1) = -(r-2).
\end{align*}

\end{proof}

\cb{\subsection{Connections to the distance hierarchy}\label{sec:hierarchy}

A broad overview of many results in the field of combinatorial optimization is~\cite{Deza.Laurent.book} (see also references therein).  The topic of that book is the study of distance spaces and their ability to be embedded in other fixed structures.  Every graph can be thought of as a distance space in a natural manner by defining the distance between two vertices to be their path distance in the graph; as such, results about distance spaces are equally applicable to graphs.  One of the main theoretical structures laid out there is the distance hierarchy, which we explain briefly below.  First, some definitions:

\begin{define}
  We say that a graph $G$ is {\bf hypercube embeddable} if it can be isometrically embedded in the Hamming graph $H(m,2)$ for some $m\ge1$; alternatively, $G$ is hypercube embeddable if it is possible to assign a binary string to every node of the graph so that the Hamming distance between the strings is the same as the graph distance.  
    A graph $G$ is {\bf $\ell^1$-embeddable} if it can be isometrically embedded into the space $(R^m,\ell^1)$ for some $m\ge 1$; alternatively, $G$ is $\ell^1$-embeddable if we can assign a vector in $\R^m$ to each vertex of the graph in such a manner that the $\ell^1$ distance between the vectors is the same as the graph distance.
\end{define}

For the purposes of this paper, the metric hierarchy says the following~\cite[Section 19.2]{Deza.Laurent.book}: 

\begin{itemize}
\item $G$ is hypercube embeddable $\implies$ $G$ is $\ell^1$-embeddable $\implies$ the classical distance matrix of $G$ has one positive eigenvalue;
\item in general, neither of those implications are reversible (i.e. the sets of graphs are strictly increasing, moving left to right);
\item for bipartite graphs, the hierarchy collapses and all of the implications become bidirectional.
\end{itemize}

The connection between the hierarchy and the results of this paper are summarized in the following theorem:

\begin{thm}
We have the following:
\begin{enumerate}
\item Let $G$ be bipartite and distance-regular.  If $G$ is uniformly positive definite, then it is hypercube embeddable.
\item Let $G$ be bipartite and distance-regular.  If $G$ is hypercube embeddable, then it is weakly positive definite.
\item A graph $G$ is $\ell^1$-embeddable iff it is an isomorphic subgraph of a graph $\widehat{G}$ that is uniformly positive definite.
\end{enumerate}
\end{thm}

\begin{remark}
  The class of bipartite distance-regular graphs has been characterized in~\cite{Lee.Weng.14}.
\end{remark}

\begin{proof}
  The first two statements follow from the distance hierarchy and Proposition~\ref{prop:onepositive} above.  If $G$ is bipartite, then it is hypergraph embeddable iff the classical distance matrix has one positive eigenvalue~\cite[Theorem 19.2.8]{Deza.Laurent.book}; in  Proposition~\ref{prop:onepositive} we show that one positive eigenvalue implies weakly positive definite and is implied by uniformly positive definite. 
  
    For the third, it is shown in~\cite[Theorem 21.1.3]{Deza.Laurent.book} that every $\ell^1$-embeddable graph is an isometric subgraph of a graph $\widehat{G}$ that is a finite product of graphs, each of which is the complete graph $K_n$, a cocktail party graph $K_{m\times 2}$, or a halved cube $HC_d$.   $K_n$ has a diameter of one, and if we set $f_0=1, f_1=z$ then $\mc{M}(f;K_n) = (1-z)I + zJ$, where $J$ is the all-ones matrix.  From this we can compute directly that $\phi_0(z) = 1+(n-1)z$ and $\phi_i(z) = 1-z$ for all $i>0$.  As such, $K_n$ is uniformly positive definite.  The graph $K_{m\times 2}$ is a strongly regular graph with parameters $(2m,2m-1,2m-2,2m-2)$, and is thus uniformly positive definite by the results of Section~\ref{sec:strong}.  Finally, we show in Section~\ref{sec:half} that $HC_d$ is uniformly positive definite.  From this, and Theorem~\ref{thm:sum}, we obtain the result.

\end{proof}

}

\section{Applications}

\subsection{Generalized Lotka--Volterra competition model}\label{sec:GLVC}

\subsubsection{Background}

There are a variety of models of ecosystem dynamics in the literature, which include Eigen's quasispecies model~\cite{Crow.Kimura.book, Domingo.Schuster.book, Roughgarden.book, Malarz.Tiggemann.98}, also known as the replication-mutation equation~\cite{Sigmund.Hofbauer.book}, as well as various competition-utilization models~\cite{Schoener.76, Turchin.book, Pianka.book}.  The model we address here is the generalized Lotka--Volterra competition (GLVC) model, described below; note that GLVC can be shown to be equivalent to many of the other common models used in ecosystem dynamics~\cite{Page.Nowak.02}.  The main motivation for the GLVC model is the assumption that there are $n$ species that interact through at the population level, and the rate of growth of any one species is an affine function of the population sizes of each of the other species.  More concretely, let $C$ be a symmetric $n\times n$ matrix with nonnegative entries, and $r\in \R^n$.  Then the {\bf competition model without mutation} is given by
\begin{equation}\label{eq:wo}
  \dfrac{d}{dt} x_i = x_i\left(r_i - \sum_{j} c_{ij} x_j\right).
\end{equation}

\cb{The interpretation of this model and the parameters are as follows.  The quantity $x_i\ge 0$ corresponds to the size, or in some cases the concentration, of species $i$.   First note that if $x_i = 0$ then $dx_i/dt= 0$, meaning that if a species goes extinct, it stays extinct.  The term in the parentheses is the rate of growth of species $i$ at any time; if it is negative the population will decay to zero and if positive the population will grow.  By assumption, the term that represents the impact of species $j$ on species $i$ is $-c_{ij}x_j\le 0$, which is always nonpositive.  Moreover, it is only zero if $x_j=0$ (species $j$ is extinct) or if $c_{ij} = 0$ (species $j$ does not impact species $i$).  This is why the model is called a competition model, since species interact only through suppressing each other.  One final note:  the interaction term is always assumed to be quadratic, which is plausible since the number of interactions between species $i$ and species $j$ will be linear in each of the population sizes.}

\cb{The system~\eqref{eq:wo} can be more compactly written as $x' = x\odot (r-Cx)$, where $\odot$ represents the pointwise product of vectors.} Note that if $Cx^* = r$, then $x^*$ is a fixed point of~\eqref{eq:wo}.  It is known~\cite{Takeuchi.book} for this model that if $C$ is symmetric, and $x^*$ is a locally attracting interior fixed point (i.e. that $x_i^*>0$ and the Jacobian at $x^*$ is negative semidefinite), then $x^*$ is the unique fixed point for~\eqref{eq:wo} in the positive octant and, moreover, is globally attracting.

We compute that the Jacobian of the vector field in~\eqref{eq:wo} at $x^* = {\bf 1}$ is $-C$.  Therefore, if we choose $r=C{\bf 1}$, then~\eqref{eq:wo} has a globally attracting fixed point at ${\bf 1}$ iff $C>0$.  Moreover, with a bit more work we can determine that the Jacobian at $x^*$ is the matrix $J=-\diag(x^*)\cdot C$\cb{, where $\diag(x^*)$ represents the diagonal matrix with $x^*_i$ in the $(i,i)$th location}.  Note that if $x_i^*>0$ for all $i$, then $J$ is positive-definite iff $C$ is.  Thus, a more general construction is:  choose any $x^*$ in the positive octant and let $r=Cx^*$, then this point is globally attracting iff $C>0$. In short, if we can show that $C$ is positive-definite, then we understand the global dynamics of~\eqref{eq:wo} completely. One can also consider the competition model {\em with mutation}:
\begin{equation}\label{eq:w}
  \dfrac{d}{dt} x_i = r_i \sum_j d_{ij} x_j  - x_i\sum_{j} c_{ij} x_j
\end{equation}
where we assume that the matrix $D$ has row sums all equal to 1 (i.e. D{\bf 1} = {\bf 1}).  We see that $x^*$ is a fixed point for this system if $r_i d_{ii} = \sum_j c_{ij} x_j^*$, and the Jacobian at $x^*$ is $\diag(r)(D-I)-\diag(x^*)C$.  If we assume that $d_{ii}$ is independent of $i$ and $C$ has constant row sum, then $r$ is a constant vector and the Jacobian at $x^* = \bf 1$ is just $r(D-I)-C$.  Under some quite mild assumptions (e.g. $d_{ij}\ge 0$) we see that $D-I$ is negative semidefinite, so if $C,D$ commute then we see that $C>0$ is again a sufficient condition for stability, i.e. ``mutation cannot hurt, it can only help''.  

In summary, the point $x={\bf 1}$ is (asymptotically) stable under~\eqref{eq:wo} if $C$ is positive definite, and it is (asymptotically) stable under~\eqref{eq:w} if $r(D-I)-C$ is negative definite.

\subsubsection{The GLVC model under assumptions of graph regularity}

Now let us assume that the species in our model interact according to some graph topology, by which we mean:  \cb{we assume that there is a graph $G=(V,E)$ with $\av{V}=n$,} that the strength of the interaction between species $x_i$ and $x_j$ is a function of the distance between vertex $i$ and vertex $j$ in the graph and, if mutation is present, the probability of species $i$ mutating to species $j$ is also a function of the distance between vertex $i$ and vertex $j$ in the graph.  This implies that $C$ in~\eqref{eq:w} and $C,D$ in~\eqref{eq:wo} are generalized distance matrices for the underlying graph, i.e. $C=\mc{M}(f;G)$ and $D=\mc{M}(g;G)$ for some $f,g$.  Therefore the question of stability for such systems is one of the spectrum of generalized distance matrices.

There has been a significant study of the GLVC equations going back to~\cite{Macarthur.Levins.67} with models sharing the property that the strength of interaction between any two species in the system is a function of their distance in some sort of ``feature space''~\cite{Galluccio.97, Doebeli.Dieckmann.00, Scheffer.vanNes.06, Lawson.Jensen.07, Pigolotti.Lopez.Hernandez-Garcia.07, Fort.Scheffer.vanNes.09, Pigolotti.etal.10, Fort.Scheffer.vanNes.10, Park.Munoz.Deem.10, Fort.Inchausti.12}.  When the feature space is taken to be discrete, it is common to use a graph consistent the Hamming distance on a set of sequences~\cite{Altmeyer.McCaskill.01, Rogers.McKane.Rossberg.12, Biancalani.DeVille.Goldenfeld.15, Saakian.etal.17}.  In particular, one obvious choice is to assume that a pair of  species interacts in a way that is a function of their genomes. If we can further assume that the interaction strength is a function only of the number of loci at which the genome differ, then the underlying graph topology is the Hamming graph $H(n,4)$ with alphabet $\{\mathtt{C,G,T,A}\}$ and thus the matrix $A$ is a generalized distance matrix for $H(n,4)$.  Moreover, if we assume that all mutations are ``point mutations'', i.e. occur at loci independently with a fixed probability, then $B$ is also a generalized distance matrix for $H(n,4)$.  More generally, one can consider phenotypic ``niche'' models on more general graphs constructed as the Cartesian product of a sequence of complete graphs.  See~\cite{Semenov.Novozhilov.17} for a recent algebraic approach to this problem.

\cb{Recall Definition~\ref{def:cd} for the competition domain and for uniform and weak positive definiteness. The motivating biological principle behind the definition is that, assuming graph distance represents the degree of dissimilarity between species, then species that are more alike should compete more strongly.  The fact that all species are competing means that $f_i\ge 0$, but the condition $f_i \ge f_{i+1}$ quantifies the fact that there is stronger competition between species that are more alike.  Scaling all of the coefficients in~~\eqref{eq:wo} or~\eqref{eq:w} corresponds to a rescaling of time and as such without loss of generality we can assume that $f_0=1$.  Finally, the choice $f_m = z^m$ with $z\in[0,1]$ is in the competition domain, and has the interpretation that the competition strength is multiplicative in distance, i.e. so each link in a chain between two species attenuates the competition strength by a common factor.  Under this interpretation, the notion of uniformly positive definite means that~\eqref{eq:wo} is stable for any level of attenuation.  The notion of weakly positive definite means that~\eqref{eq:wo} is stable when the attenuation is not too strong --- in short is stable under ``strong competition''.  In this context, the Propositions~\ref{prop:onepositive},\ref{prop:minmax}, and~\ref{prop:min} give conditions for stability of~\eqref{eq:wo}.
}

{

\begin{prop}
   Let $C = \mc{M}(f;G)$ for any $f$ in the competition domain.  Then there is a choice of parameters $g$ such that $D = \mc{M}(g;G)$ is a mutation matrix and $x = {\bf 1}$ is asymptotically stable under~\eqref{eq:w}.
\end{prop}

\begin{proof}
  This follows if we can show that $r(D-I)-C$ is negative definite.  Since $C,D$ commute, we can just add eigenvalues.  Proposition~\ref{prop:min} implies that the eigenvalues of $-C$ are all less than $r-2$.  We can choose $g$ in such a way that $D$ is a constant matrix with row sum one, so in fact is $n^{-1}J$.  Thus the eigenvalues of $D-I$ are zero with multiplicity one and $-1$ with multiplicity $(n-1)$.  Also, note ${\bf 1}$ is in the nullspace of $D-I$.  Thus we have $(r(D-I)-C){\bf 1} = -r{\bf 1}$, and for any other eigenvector $v_i$, we have $(r(D-I)-C)v \le -r+(r-2) = -2$.  
\end{proof}

In short, this shows that no matter how we choose parameters in the competition domain, there is some choice of mutation that stabilizes the system --- in short, diffusion can smooth out any nonlinearity here.  The natural question would be how to determine the minimal amount of mutation necessary to smooth out a given nonlinear instability.   For example, let us imagine that we have an $f$ in the competition domain such that $\lambda_i(C) = \lambda_i(f;G)<0$.  We then have, for $i>0$, 
\begin{equation*}
  \lambda_i(r(D-I)-C) = \sum_{m=0}^d \lambda_{i,m}(rg_m-f_m)-r,
\end{equation*}
and we want to choose $g$ so that this is negative.  Clearly for any $\lambda_{i,m}>0$ we can choose $g_m = 0$, and thus to make this eigenvalue negative we could minimally choose $g_m>0$ only for those $\lambda_{i,m}<0$.  From the Corollary above we can always do this simultaneously for each $i>0$.

\subsection{Rapidly-mixing Markov chains}

\subsubsection{Background}

Given a graph $G$, there are a variety of ways~\cite{Lovasz.93, Norris.book} to define a Markov chain corresponding to a random walk on $G$.  In some sense, the notion that the random walk be consistent with the graph $G$ is a {\em restriction} of possible transitions --- the state space of the Markov chain is the vertices of the graph, and the allowable transitions are those that take place along edges of the graph, or perhaps only along paths shorter than a given fixed distance.  
\newcommand{\lmax}{\nu_{\mathsf{max}}}

The problem we consider is this.  Let $G$ be a graph of diameter $d$, and choose $d' \le d$.  We consider random walks that can take steps of size $d'$ or less on the graph $G$ (see~\cite{Gunes.Spaniol.02, Avin.Krishnamachari.08, Beraldi.09,  Shakkottai.05} for applications and also~\cite{Durrett.Kesten.Limic.02, Holmes.Sakai.07}), and ask how to choose the transition probabilities in such a manner that the random walk decays to equilibrium most quickly.   There is a significant literature on this problem in the case where $d'=1$ which corresponds to restricting that jumps take place only on edges (see~\cite{Sinclair.Jerrum.89, Dyer.Greenhill.98, Boyd.etal.04, Carli.etal.08} for this and closely related problems, the closest to our approach here being the considerations of graphs with symmetries in~\cite{Boyd.etal.09}).  \cb{We will focus on discrete-time walks in this section, but see Remark~\ref{rem:CTMC} for comments on continuous-time random walks.}

Given an $n\times n$ symmetric matrix $P$ with non-negative entries and row sums all one, the (discrete-time) Markov chain generated by $P$ is the stochastic process $(X_t)_{t=0}^\infty$ defined by 
\begin{equation*}
  \P(X_{t+1} = j|X_t = i) = P_{ij}.
\end{equation*}
(For the stochastic process to be well-defined we have to specify the initial distribution of $X_0$.)  Since $P$ is symmetric, ${\bf 1}$ is both a right and a left eigenvector and therefore the invariant distribution is $n^{-1}{\bf 1}$.  The next question then is:  how quickly does a typical initial condition decay to the invariant distribution?  As is well-known~\cite{Levin.Peres.Wilmer.book}, this can be answered if we know the spectrum of $P$. Let us\footnote{It is common in this context to use $\lambda$ to denote the eigenvalues of $P$, but we have another use for $\lambda$ below.} write the eigenvalues of $P$ as $\nu_i(P)$.  Note that $\nu_i(P)$ are real since $P$ is symmetric, and moreover they must lie within $[-1,1]$ by the Perron-Frobenius Theorem~\cite[Section 8.4]{Horn.Johnson.book}.  We number the eigenvalues as
\begin{equation*}
  1 = \nu_1(P) \ge \nu_2(P) \ge \dots \ge \nu_n(P) \ge -1,
\end{equation*}
and then the eigenvalue with maximal modulus is the one which determines the decay rate to equilibrium.  That is to say, if we define 
\begin{equation*}
  \lmax(P) := \max_{i=2,\dots,n}\av{\nu_i(P)} = \max\{\nu_2(P),-\nu_n(P)\},
\end{equation*}
 and if the distribution of $X_0$ is $\alpha_0$, and the distribution of $X_t$ is $\alpha_t$, then for almost all $\alpha_0$, the distribution decays to equilibrium at rate $(\lmax)^t$, or, more precisely,
\begin{equation}
 \lim_{t\to\infty}\left(\norm{\alpha_t - n^{-1}{\bf 1}}\right)^{1/t} = \lmax.
\end{equation}
In particular, when $\lmax$ is close to zero, this means initial distributions decay very quickly to equilibrium, but when $\lmax$ is  close to $1$ this means they decay slowly.  Thus to get ``fast mixing'' we want to find the smallest possible $\lmax$ where we are allowed to vary the transition rates in some manner.  Typically, we refer to the ``spectral gap'' of $1-\lmax$, and thus fast decay is equivalent to a large spectral gap.  By definition, the spectral gap lies between zero and one.  (See also \cite{Bubley.Dyer.97, Morris.Peres.05, Chen.Saloff-Coste.13, Goel.Montenegro.Tetali.06, Montenegro.Tetali.06}.)

\subsubsection{Rapid mixing for distance-regular graphs}

\begin{define}
  { Following~\cite{Hilano.Nomura.84}, we define a {\bf degree-regular} graph to be} a graph where $\av{G_k(x)}$ is independent of $x$.  We write $n_k = \av{G_k(x)}$ for such a graph.  
\end{define}

Clearly distance-regular graphs are {degree-regular}, but the converse is false, e.g. $K_2\cpg K_3$.

\begin{define}
Let $G$ be a {degree-regular} graph with diameter $d$ and let $\mu = (\mu_0,\mu_1,\dots,\mu_d)$ such that
\begin{equation*}
  \sum_{i=0}^d \mu_i = 1.
\end{equation*}
  We define the (transition matrix of the) {\bf multi-step (discrete-time) Markov chain (DTMC) on the graph $G$ with transition probabilities $\mu$} as the $\av{V(G)}\times\av{V(G)}$ matrix $P$, where  
\begin{equation*}
  P_{xy} = \dfrac{\mu_{d(x,y)}}{n_{d(x,y)}}.
\end{equation*}

\end{define}

The interpretation of the DTMC is that the parameter $\mu_k$ defines the probability of taking a step of length $k$, and then we assume that all possible steps of length $k$ are chosen equally likely.
Choosing $\mu$ determines $f_k = \mu_k/n_k$.  Recall~\eqref{eq:lambdaf} and writing $C_{i,m} = p_m(\lambda_i)$, we have
\begin{equation*}
  \lambda_i = \sum_{k=0}^d C_{i,k}f_k  = \sum_{k=0}^d \dfrac{C_{i,k}}{n_k}\mu_k.
\end{equation*}
Noting that $C_{i,0} = n_0 = 1$, and using the constraint $\sum\mu_k = 1$, we have
\begin{equation*}
  \lambda_i = \mu_0 + \sum_{k=1}^d \dfrac{C_{i,k}}{n_k}\mu_k = 1 +\sum_{k=1}^d \left(\dfrac{C_{i,k}}{n_k}-1\right)\mu_k.
\end{equation*}
Let us write $D_{i,k} = C_{i,k}/n_k - 1$.  Note that $C_{i,k} \le n_k$ for all $i,k$, and therefore $D_{i,k}\le 0$, and we have the following  linear programming problem:

{\bf Discrete Optimization Problem.}  For a given graph $G$ and $d' \le d = \diam(G)$, minimize
\begin{equation*}
  \max_i \av{\lambda_i} = \max_i \av{1+\sum_{k=1}^{d'} D_{i,k}\mu_k} = \max_i \max\left\{1+\sum_{k=1}^{d'} D_{i,k}\mu_k,-\left(1+\sum_{k=1}^{d'} D_{i,k}\mu_k\right)\right\}
\end{equation*}
over the set $\mu_1+\dots\mu_{d'} \le 1$.  We denote the solution of this problem by $\mc{D}(G,d')$.

Note that the complexity of this problem is relatively small to a naive approach, especially for graphs with small diameter but large order: we have to find the minimizer of at most $2d$ linear functions of $d-1$ independent variables but the order of the graph does not appear.  This can of course be attacked by the standard methods~\cite{Boyd.Vandenberghe.book} and for small $d$, as we see below, we can even write down the solutions in more or less closed form.
The first thing to see is that if we are allowed to choose all of the $\mu_k$ positive, then there is a universally optimal solution to the discrete problem.

\begin{prop}
  Let $G$ be a graph with $\diam(G) = d$.  If we choose $\mu$ proportional to the vector $(n_0,n_1,\dots, n_d)$, then $\mc{D}(\mu;G) = 0$, then $\lambda_i = 0$ for all $i>0$, and thus $\mc{D}(G,d) = 0$, giving the maximal spectral gap of one.
\end{prop}

\begin{proof}
If we choose $\mu$ proportional to $n$, then this means that $f_k$ is independent of $k$, and therefore $\mc{M}(f;G)$ is a constant multiple of the all-ones matrix $J$.  As such it has one positive eigenvalue and the remainder zero.
\end{proof}

\cb{
\begin{remark}[Continuous-time Markov chains]\label{rem:CTMC}
We have considered discrete-time Markov chains above, but we could have just as easily considered a continuous time Markov chain, as follows:
Let $\rho= (\rho_0,\rho_1,\dots,\rho_d)$ with $\rho_i\ge 0$. Then define the (generator of the) {\bf multi-step (continuous time) Markov chain on the graph $G$ with transition rates $\rho$} as the $\av{V(G)}\times\av{V(G)}$ matrix $Q$, where  
\begin{equation*}
  Q_{xy} = \begin{cases}\dfrac{\rho_{d(x,y)}}{n_{d(x,y)}}, & x\neq y\\\\-\sum_{m=1}^d \rho_m,&x=y.\end{cases}
\end{equation*}
We then define a stochastic process on $\{1,\dots,\av{V(G)}\}$ infinitesimally: $\P(X_{t+h} = y | X_t = x) = Q_{xy}$.  In this case we can compute the evolution of probabilities in closed form:  if $p(t)$ is the vector of probabilities, i.e. $p_i(t) = \P(X_t = i)$, then (recalling that $Q$ is symmetric and following~\cite{Norris.book}) $p(t) = e^{tQ}p(0)$.  In this case, we have that the spectrum of $Q$ lies in $(-\infty,0]$ and again to obtain rapid mixing we need to control the right-most eigenvalue, in a similar fashion to the discrete optimization problem above.

\end{remark}
}

\section{Examples}\label{sec:examples}

Here we compute many of the quantities discussed above for various families of graphs.

\subsection{Strongly regular graphs}\label{sec:strong}

When a distance-regular graph has diameter two, it is called strongly regular.  These graphs are well-studied to the point that they have their own notation which we introduce now.

\begin{define}
We call a graph $G$ {\bf strongly regular} with parameters $(n,k,\alpha,\beta)$ if $G$ has $n$ vertices, valency $k$ and has the property that whenever two vertices are adjacent, they have $\alpha$ neighbors in common, and whenever they are not adjacent they have $\beta$ neighbors in common.  Clearly the graph is connected iff $\beta>0$, and we assume this throughout.  If $\beta>0$, then clearly $\diam(G) =2$.
\end{define}

We can compute~\cite{Biggs.AGT} that a strongly regular graph with parameters $(n,k,\alpha,\beta)$ has intersection array $\{k,k-\alpha-1;1,\beta\}$ and thus $Q$ matrix
\begin{equation*}
  Q = \left(\begin{array}{ccc} 0&k&0\\1&{\alpha}&k-\alpha-1\\0&\beta&k-\beta\end{array}\right).
\end{equation*}
Moreover, the eigenvalues of $Q$ are $k$ and 
\begin{equation*}
  \dfrac12\left((\alpha-\beta) \pm \sqrt{(\alpha-\beta)^2 + 4(k-\beta)}\right).
\end{equation*}
Typically the larger of these is called $\theta$ and the smaller $\tau$.  The first question we might ask is which strongly regular graphs are uniformly stable or not.  We first have the following proposition:

\begin{prop}
  If $\diam(G) = 2$, then $\phi_i(z;G)\ge 0$ for all $z\in[0,1]$ if and only if $\lambda_i \ge -2$, and $\phi_i'(z;G)|_{z=1} = 0$ iff $\lambda_i = -2$.
\end{prop}

\begin{proof}
  By definition we have
\begin{equation*}
  \phi_i(z) = 1 + \lambda_i z + p_2(\lambda_i)z^2,
\end{equation*}
and using Proposition~\ref{prop:factor} this means that
\begin{equation}\label{eq:phi2}
  \phi_i(z) = (1-z)(1+(\lambda_i+1)z).
\end{equation}
If $\lambda_i \ge -2$, then $\lambda_i+1 \ge -1$ and neither of those factors is zero inside $(0,1)$.  Finally, if $\lambda_i  =-2$ then~\eqref{eq:phi2} is $(z-1)^2$ and has a double root at $z=1$.
\end{proof}

From this it follows that $\Phi(z;G) \ge 0$ for all $z\in[0,1]$ iff $\tau \ge -2$.  If $\tau < -2$, then we have that $\Phi(z;G)\ge 0$ for $z\in[0,-1/(\tau+2)]$.  Also, note that this implies that the classical distance matrix has zero eigenvalues iff $\tau = -2$.  See, for example,~\cite{Seidel.68}.
With some algebra, we see that the condition $\tau\ge -2$ is equivalent to the condition
\begin{equation}\label{eq:cond}
  k - 2\alpha + \beta \le 4.
\end{equation}

\subsection{Taylor graphs}

Taylor graphs are those graphs of diameter three with intersection array $\{k,\mu,1;1,\mu,k\}$.  In this case we have
\begin{equation*}
  Q = \left(
\begin{array}{cccc}
 0 & k & 0 & 0 \\
 1 & k-\mu -1 & \mu  & 0 \\
 0 & \mu  & k-\mu -1 & 1 \\
 0 & 0 & k & 0 \\
\end{array}
\right),
\end{equation*}
but it is also not hard to compute that 
\begin{equation*}
  Q_2 = \left(
\begin{array}{cccc}
 0 & 0 & k & 0 \\
 0 & \mu  & k-\mu -1 & 1 \\
 1 & k-\mu -1 & \mu  & 0 \\
 0 & k & 0 & 0 \\
\end{array}
\right),\quad Q_3 = \left(
\begin{array}{cccc}
 0 & 0 & 0 & 1 \\
 0 & 0 & 1 & 0 \\
 0 & 1 & 0 & 0 \\
 1 & 0 & 0 & 0 \\
\end{array}
\right).
\end{equation*}
It is clear that the eigenvectors of $Q_3$ are either palindromes or anti-palindromes (specifically, palindromes satisfy $Qv=v$ and anitpalindromes satisfy $Qw=-w$).  Let us define $\R^4 = V\oplus W$ where $V$ is the palindromic subspace and $W$ the antipalindromic.  By inspection $Q_2=Q_1Q_3$, and thus if $\lambda_{2,i} = \lambda_{1,i}$ if $v\in V$ and $\lambda_{2,i} = -\lambda_{1,i}$ if $v\in W$.  Let us assume that $v_0,v_1\in V$ and $v_2,v_3\in W$.  Then for $k=0,1$ we have
\begin{equation*}
  \phi_{k}(z) = 1 + \lambda_{k}z + \lambda_k z^2 +z^3, 
\end{equation*}
and for $k=2,3$ we have
\begin{equation*}
\phi_k(z) = 1 + \lambda_k z - \lambda_k z^2 - z^3. 
\end{equation*}
Moreover, we see that $Q|_V = Q_s\otimes Q_s$ and $Q|_W = Q_u\otimes Q_u$ where
\begin{equation*}
  Q_s = \left(
\begin{array}{cc}
 0 & k  \\
 1 & k-1\\
\end{array}
\right),\quad 
  Q_u = \left(
\begin{array}{cc}
 0 & k  \\
 1 & k-2\mu-1\\
\end{array}
\right).
\end{equation*}
From this we obtain $\lambda_0 = k, \lambda_1 = -1$ a the eigenvalues of $Q_s$.  The other two eigenvalues are the eigenvalues of $Q_u$, which are 
\begin{equation}\label{eq:defoftheta}
  \theta_\pm = \dfrac12\left((k-2\mu-1)\pm \sqrt{(k-2\mu-1)^2+4k}\right),
\end{equation}
and thus we have
\begin{align*}
  \phi_0(z) &= 1+kz+kz^2+z^3,\\
  \phi_1(z) &= 1-z-z^2+z^3 = (1-z)^2(1+z),\\
  \phi_2(z) &= 1+\theta_+z - \theta_+z^2 - z^3 = (1-z)(1+(\theta_++1)z+z^2),\\
  \phi_3(z) &=  1+\theta_-z - \theta_-z^2 - z^3 = (1-z)(1+(\theta_-+1)z+z^2).
\end{align*}

Thus $G$ is uniformly positive definite iff the quadratic  $1+(\theta_\pm+1)z+z^2\ge 0$ for all $z\in[0,1]$.  This quadratic is nonnegative on the interval if and only if the coefficient of the linear term is $\ge -2$, and so therefore we have the condition $\theta_\pm \ge -3$.  Since $\theta_+ \ge \theta_-$, this means that $\theta_- \ge -3$.  Using~\eqref{eq:defoftheta} plus the condition $0<\mu<k$ this reduces to the condition $k \ge 3(\mu-1)$.  In particular, the critical value $k = 3(\mu-3)$ corresponds to the case where $\phi_3(z)$ has a triple root at $z=1$ (these graphs are discussed in~\cite[Corollary~1.15.3]{Brouwer.Cohen.Neumaier.book}).  Some concrete examples include the halved 6-cube ($k=15, \mu = 6$) and the Gosset graph ($k=27,\mu=10$).

One large class of Taylor graphs are the crown graphs (an $n$-crown graph can be defined as the graph complement of $K_n\cpg K_2$), and in this case $k = n-1, \mu = n-2$.  The sequence is usually taken to start at $n=3$ (which graph is actually $C_6$).  From this, we see the condition to be uniformly positive definite is $n-1\ge 3(n-2)$ or $n \le 7/2$.  This means that for $n\ge 4$ the $n$-crown graph is not uniformly positive definite.

There are many interesting results about Taylor graphs, and examples thereof, in~\cite{Brouwer.Cohen.Neumaier.book}.  For example, it is known that if $G$ is Taylor with parameters $k,\mu$, then the $2$-path graph $G_2$ is Taylor with parameters $k,\tilde\mu = k-\mu-1$.  We see from the formulae above that if we let $\tilde\theta_\pm$ be the associated antisymmetric eigenvalues of $G_2$, then $\tilde\theta_\pm = -\theta_\mp$.

\subsection{Hamming graphs}\label{sec:Hamming}

The Hamming graph $H(d,q)$ with $d\ge 1, q \ge 2$ is the graph with vertex set $\{0,1,\dots,q-1\}^d$ where we say two vertices are adjacent if they differ in exactly one component.  These graphs are distance-regular with diameter $d$.  Let us first note that $K_q$ has two eigenvalues:  $\phi_0(z;K_q) = 1+(q-1)z$ with multiplicity one, and $\phi_1(z;K_q) = 1-z$ with multiplicity $q-1$.  More compactly,
\begin{equation*}
  \Phi(z;K_q) = \left(\begin{array}{cc} (1-z) & (1+(q-1)z)\\q-1&1\end{array}\right)
\end{equation*}

Since $H(d,q) = \bigbox_{i=1}^d K_q$, it follows from Corollary~\ref{corr:phi} that the components of $\Phi(z;G)$ are
\begin{equation*}
  (1-z)^d,\quad (1-z)^{d-1}(1+(q-1)z),\quad (1-z)^{d-2}(1+(q-1)z)^2,\dots,  (1+(q-1)z)^d,
\end{equation*}
with multiplicities $(q-1)^d, (q-1)^{d-1},\dots, (q-1),1$.  More compactly, we can write:
\begin{equation*}
  \Phi(z;H(d,q)) = \left(\begin{array}{c}(1-z)^m(1+(q-1)z)^{d-m}\\(q-1)^m\end{array}\right)_{m=0}^d,
\end{equation*}
From this we can deduce that $H(d,q)$ is uniformly stable for all $d,q$, i.e. that $\Phi(z;G) \ge0$ for all $z\in[0,1]$.

We can then ask the question of the optimal choice of $\mu$ to maximize mixing on this graph.  We make a surprising observation here:

\begin{obs}
  Let $q=2$ and consider the graph $G = H(d,2)$. For $d' \le d/2$ the solution to $\mc{D}(H(d,2),d')$ is the ``top two'' solution where we choose $\mu_k = 0$ for all $k<d'-1$, and 
  \begin{equation*}
    \mu_k = \begin{cases} \dfrac{d'}{d+1}, & k=d'-1,\\1-\dfrac{d'}{d+1}, & k=d',\\0,&\mbox{else.}\end{cases}
  \end{equation*}
 \cb{ We give a numerical example of this phenomenon in Figure~\ref{fig:Hamming}.}
\end{obs}

\begin{figure}[th]
\begin{equation*}
{\small
\left(
\begin{array}{cccc}
 d'=1&\dfrac{7}{9} & \left\{\dfrac{1}{9},\dfrac{8}{9}\right\} & \left\{\dfrac{1}{9},\dfrac{1}{9}\right\} \\\\
 2&\dfrac{5}{9} & \left\{0,\dfrac{2}{9},\dfrac{7}{9}\right\} & \left\{0,\dfrac{1}{36},\dfrac{1}{36}\right\} \\\\
 3&\dfrac{1}{3} & \left\{0,0,\dfrac{1}{3},\dfrac{2}{3}\right\} & \left\{0,0,\dfrac{1}{84},\dfrac{1}{84}\right\} \\\\
 4&\dfrac{1}{9} & \left\{0,0,0,\dfrac{4}{9},\dfrac{5}{9}\right\} & \left\{0,0,0,\dfrac{1}{126},\dfrac{1}{126}\right\} \\\\
 5&\dfrac{5}{123} & \left\{\dfrac{1}{123},\dfrac{8}{123},\dfrac{8}{123},\dfrac{16}{123},\dfrac{50}{123},\dfrac{40}{123}\right\} & \left\{\dfrac{1}{123},\dfrac{1}{123},\dfrac{2}{861},\dfrac{2}{861},\dfrac{5}{861},\dfrac{5}{861}\right\} \\\\
 6&\dfrac{1}{69} & \left\{0,\dfrac{5}{138},\dfrac{35}{276},\dfrac{14}{69},\dfrac{35}{138},\dfrac{35}{138},\dfrac{35}{276}\right\} & \left\{0,\dfrac{5}{1104},\dfrac{5}{1104},\dfrac{1}{276},\dfrac{1}{276},\dfrac{5}{1104},\dfrac{5}{1104}\right\} \\\\
 7&\dfrac{1}{255} & \left\{\dfrac{1}{255},\dfrac{8}{255},\dfrac{28}{255},\dfrac{56}{255},\dfrac{14}{51},\dfrac{56}{255},\dfrac{28}{255},\dfrac{8}{255}\right\} & \left\{\dfrac{1}{255},\dfrac{1}{255},\dfrac{1}{255},\dfrac{1}{255},\dfrac{1}{255},\dfrac{1}{255},\dfrac{1}{255},\dfrac{1}{255}\right\} \\\\
 8&0 & \left\{\dfrac{1}{256},\dfrac{1}{32},\dfrac{7}{64},\dfrac{7}{32},\dfrac{35}{128},\dfrac{7}{32},\dfrac{7}{64},\dfrac{1}{32},\dfrac{1}{256}\right\} & \left\{\dfrac{1}{256},\dfrac{1}{256},\dfrac{1}{256},\dfrac{1}{256},\dfrac{1}{256},\dfrac{1}{256},\dfrac{1}{256},\dfrac{1}{256},\dfrac{1}{256}\right\} \\
\end{array}
\right) }
\end{equation*}
\caption{\cb{Parameters for optimally mixing Markov Chain for $H(8,2)$ and $d'=1,\dots, 8$.  In each row we plot the top eigenvalue, the vector $\mu$, and the vector $f$ where $f_k=\mu_k/n_k$. }}\label{fig:Hamming}
\end{figure}

We can also think about the GCLV problem posed above in~\eqref{eq:w}.  Recall that we are interested in showing that the matrix $ r(D-I)-C$ is negative semidefinite when $C = \mc{M}(f;H(d,q))$ and $D=\mc{M}(g;H(d,q))$ and $r$ is the row sum of $C$, where $f$ is (for example) chosen in the competition domain and $g$ has the property that it comes from mutation, so that $B$ has row sum one.

One particular example for $q=2$ is the case where $g_k = \mu^k(1-\mu)^{d-k}$, which corresponds to independent point mutations on a binary sequence.  We saw above that for any choice of $f$, there is a choice of $\mu$ that makes $r(D-I)-C$ negative semidefinite.  In particular, choosing $\mu = 1/2$ gives $g_k = 2^{-d}$ for all $k$, so that $\mc{M}(g,H(d,2))$ is $2^{-d}J_{2^d}$.  As such, $\mc{M}(g,H(d,2))$ has eigenvalues $1$ with multiplicity one and $0$ with multiplicity $2^d-1$.  By (for example) Perron--Frobenius, we know that all of the eigenvalues of $C$ aside from $r$ have modulus strictly less than $r$, and thus all of the eigenvalues of $r(D-I)-C$ are strictly negative when $\mu=1/2$.

However, we know that for $\mu=0$ this system is not negative semi-definite in general when $f$ is allowed to range over the competition domain.  The techniques presented in this paper allow us, for example, to compute exactly the set of $\mu$ that stabilizes the matrix $C$ (see for example~\cite{Rogers.McKane.Rossberg.12} where such a computation is done numerically).

In a similar fashion, we can consider the generalized sequence graph $G = \bigbox_{i=1}^q K_{n_i}$.  Note that this graph is not distance-regular, but is the Cartesian product of a family of distance-regular graphs, and thus $\Phi(z;G)$ has at most $2^q$ distinct values: 
\begin{equation*}
 \{(1-z)^q\} \cup \bigcup_{i=1}^q \{(1-z)(1+(n_i-1)z) \}\cup\bigcup_{i\neq j}^q \{(1+(n_i-1)z)(1+(n_j-1)z)\}
\end{equation*}

From this we see that these graphs are also uniformly positive definite, i.e. positive definite for all $z\in[0,1]$.  The model for GCLV with these graphs was studied in~\cite{Biancalani.DeVille.Goldenfeld.15}.

\cb{\subsection{Halved cube graphs}\label{sec:half}

Let us consider the halved cube graph $HC_d$, as follows: the vertex set consists of all binary sequences of length $d$ with even parity, and  two sequences are adjacent if they have Hamming distance equal to $2$.

We can think of this as derived from the second power of the Hamming graph $H(d,2)$ as follows:  considering the vertwe consider all paths of length two in $H(d,2)$, then this gives two connected components: the sequences of even parity and the sequences of odd parity and this gives two copies of $HC_d$. 

The latter viewpoint is useful in the following manner:  to compute $\Phi(z;HC_d)$, we consider $G = H(d,2)$ and define $f$ so that $f_{2k} = z^k$ and $f_{2k+1} = 0$.  In particular, we can use the formulas of the previous section, consider only the even terms in the expansion, and then replace each even power with its halved power.  If we denote $\phi_m(z;H(d,2)) = (1-z)^m(1+z)^{d-m}$, then we have (assuming $m\le d/2$):
\begin{align*}
  \phi_m(w;HC_d) 
  &= \frac12 \left( \phi_m(\sqrt w;H(d,2)) +\phi_m(-\sqrt w;H(d,2))\right) = \frac{(1-w)^m}2 \left((1+\sqrt w)^{d-2m} + (1-\sqrt w)^{d-2m}\right)\\
  &=(1-w)^m\sum_{k\mbox{\scriptsize{ even}}} {{d-2m}\choose k} w^{k/2}.
\end{align*}
If $m>d/2$ then we can exchange the roles of $m$ and $d-2m$ and proceed similarly. In particular, $\phi_m(w;HC_d)\ge 0$ for all $w\in[0,1]$ and therefore $HC_d$ is uniformly positive definite.

More generally, if $G$ is a distance-regular graph, and $H$, the second power of $G$ (obtained by paths of length two as above), is also a distance-regular graph, then $H$ is uniformly positive definite if $\Phi(z;G)\ge 0$ for $z\in[-1,1]$, since $\Phi(z;H)$ is, after a scaling of the independent variable, the mean $\Phi(z;G)$ for $z$ positive and $z$ negative.

}

\subsection{Johnson graphs}

The Johnson graph $J(n,d)$ is the graph whose vertices are all subsets of $[n]$ of size $d$, where we say two vertices are adjacent if they share $d-1$ elements.  This graph has diameter $d$ and intersection array given by $b_i = (d-i)(n-d-i)$ with $i=0,\dots,d-1$ and  $c_i = i^2$, $i=1,\dots, d$.

\begin{conjecture}
  If $G = J(n,d)$ for any $n\ge 2d$, then there are $d+1$ distinct eigenfunctions $\phi_0,\dots,\phi_d$ with
  \begin{equation}\label{eq:isright?}
    {\phi_m(z;G) = (1-z)^{m}\left(\sum_{k=0}^{d-m} \left(\begin{array}{c}d-m\\k\end{array}\right)\left(\begin{array}{c}n-m-d\\k\end{array}\right)z^k\right).}
  \end{equation}
\end{conjecture}

In particular, this would imply that Johnson graphs are always uniformly positive definite.

\cb{We have verified the conjecture using a computer algebra system for $d\le 6, n\ge 2d$, and we} will prove it explicitly for the case of $d=2$.  In this case, the conjecture implies that the { eigenfunctions} are
\begin{equation}\label{eq:d2}
  \phi_0(z) = 1+2(n-2)z+\dfrac{(n-2)(n-3)}2z^2,\quad \phi_1(z) = (1-z)(1+(n-3)z),\quad \phi_2(z) = (1-z)^2.
\end{equation}
The $\phi_0$ function is easy enough to verify by counting vertices in each shell.  As for the others, let us note that we have
\begin{equation*}
  Q = \left(
\begin{array}{ccc}
 0 & 2 n-4 & 0 \\
 1 & n-2 & n-3 \\
 0 & 4 & 2 (n-4) 
\end{array}
\right),\quad Q_2 = \left(
\begin{array}{ccc}
 0 & 0 &  (n-3) (n-2)/2 \\
 0 & n-3 & (n-4)(n-3)/2 \\
 1 & 2 (n-4) & (n-5)(n-4)/2
\end{array}
\right)
\end{equation*}
We can compute that the eigenvalues/vectors of $Q$ are
\begin{align*}
  \lambda_0 &= 2n-4, v_0 = (1,1,1),\\ \lambda_1 &= n-4, v_1 = (4-2n,4-n,4),\\ \lambda_2 &= -2,v_2 = ((n-3)(n-2),(3-n),2).
\end{align*}
If we plug these vectors into $Q_2$ we obtain $\lambda_{2,0} = (n-3)(n-2)/2, \lambda_{2,1} = 3-n,\lambda_{2,2} = 1$. From this we obtain the formulas in~\eqref{eq:d2} above, see also~\cite{Filmus.14}.

\begin{figure}[th]
\begin{equation*}
\left(
\begin{array}{ccc}
 d'=1 & \dfrac{31}{40} & \{0,1\}  \\ \\
 2 & \dfrac{11}{20} & \{0,0,1\}  \\  \\
 3 & \dfrac{13}{40} & \{0,0,0,1\}  \\ \\
 4 & \dfrac{1}{10} & \{0,0,0,0,1\}  \\ \\
 5 & \dfrac{29}{2351} & \left\{0,0,\dfrac{65}{2351},\dfrac{620}{2351},0,\dfrac{1666}{2351}\right\} \\ \\
 6 & \dfrac{1}{577} & \left\{\dfrac{2}{6347},\dfrac{10}{6347},\dfrac{195}{6347},\dfrac{840}{6347},\dfrac{2560}{6347},\dfrac{1440}{6347},\dfrac{1300}{6347}\right\} \\ \\
 7 & \dfrac{1}{4861} & \left\{0,\dfrac{10}{4861},\dfrac{135}{4861},\dfrac{760}{4861},\dfrac{1610}{4861},\dfrac{1596}{4861},\dfrac{630}{4861},\dfrac{120}{4861}\right\}   \\ \\
 8 & 0 & \left\{\dfrac{1}{43758},\dfrac{40}{21879},\dfrac{70}{2431},\dfrac{1120}{7293},\dfrac{2450}{7293},\dfrac{784}{2431},\dfrac{980}{7293},\dfrac{160}{7293},\dfrac{5}{4862}\right\} \end{array}
\right)
\end{equation*}
\caption{Parameters for optimally mixing Markov Chain for $J(18,8)$ and $d'=1,\dots, 8$.  In each row we plot the top eigenvalue and the vector $\mu$.  This is analogous to Figure~\ref{fig:Hamming} but we leave out the last column due to space.  Note that the last row, $d'=8$, corresponds to the uniform Markov chain, since $\mu$ is proportional to the number of vertices in each shell. }
\label{fig:Johnson1}
\end{figure}

\begin{figure}[th]
\begin{center}
\includegraphics[width=0.9\textwidth]{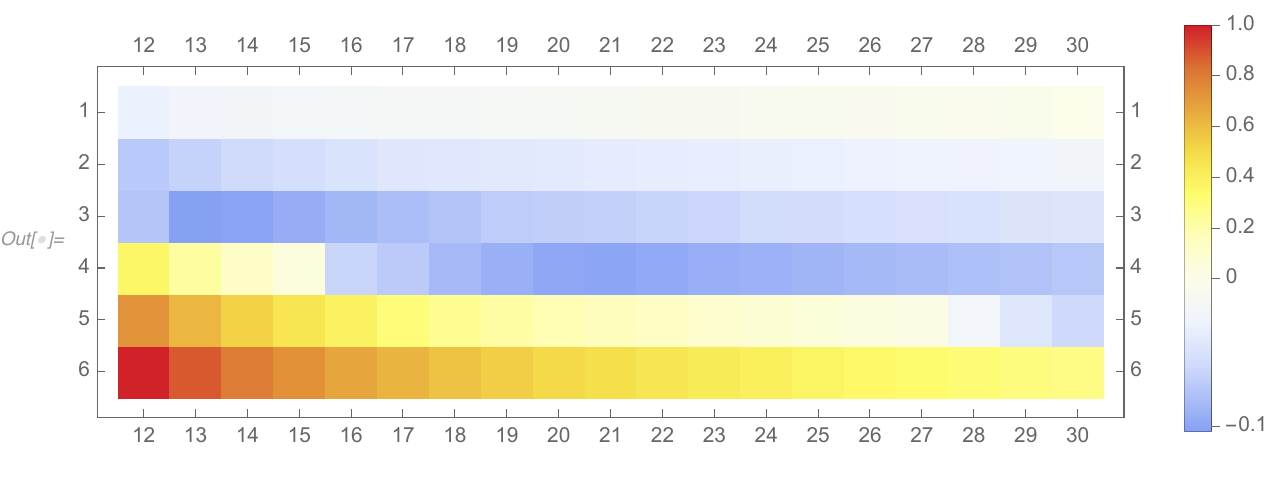}
\caption{In this plot we show the difference of the spectral gap when we choose the Markov chain that only jumps length $d'$ versus choosing uniformly on all vertices with length $\le d'$.  See text for more detail.}
\label{fig:Johnson2}
\end{center}
\end{figure}

We have also studied the optimal mixing problem on Johnson graphs as well, and we have found some interesting patterns.  For example, we have found that if we fix $d', d$ and consider the family $J(n,d)$, then as $n\to\infty$, the optimal choice of parameters to minimize $\lmax$ is to choose $\mu_{d'} = 1$ and the rest zero \cb{(i.e. $\mu = \delta_{d'}$)}.  Note that this is about as far from uniform as one might imagine, but seems to beat the uniform choice by a small amount.  \cb{We give a numerical demonstration of these observations in Figures~\ref{fig:Johnson1} and~\ref{fig:Johnson2}.  Figure~\ref{fig:Johnson1} shows the optimal choices for $J(18,8)$ and $d'=1,2,\dots,8$.  This is analogous to Figure~\ref{fig:Hamming}, but due to space we only put the $f_k$.  We see that for $d'\le 4$ the optimal choice of parameters is putting all the weight on distance $d'$.    In Figure~\ref{fig:Johnson2} we show that a similar pattern holds up for a wide range of parameters.  Here we consider $J(n,6)$ with $d'=1,2,\dots, 6$, and compute the following:  let $\alpha$ be the largest eigenvalue when we choose $\mu = \delta_{d'}$ and $\beta$ the largest when we choose every vertex with distance $\le d'$ uniformly.  In Figure~\ref{fig:Johnson2} we plot $\alpha-\beta$; when this quantity is positive the uniform Markov chain mixes faster, when it is negative the delta measure Markov chain is faster.  We see empirically that for $d'<d$, the delta measure beats the uniform measure when $n$ is large enough.  This figure only shows $d=6$ but we have empirically observed similar patterns for other $d$.

}

\subsection{Cubic graphs}

There are thirteen cubic graphs that are distance-regular, and we find $\Phi(z;\cdot)$ for eight of them in closed form.  The graphs we consider are listed in Table~\ref{table:cubic} below and appear in~\cite[Theorem 7.5.1]{Brouwer.Cohen.Neumaier.book}.  There are five more that we do not consider here in the interests of space, but they can be attacked by the techniques of this paper as well.

\begin{table}[h]
\begin{center}
\begin{tabular}{|c|c|c|}\hline
order & name &intersection numbers\\\hline
4	&$K_4$	&\{3;1\}\\\hline
6   &$K_{3,3}$ (Utility graph)	&\{3,2;1,3\}\\\hline
8	&Cube	&\{3,2,1;1,2,3\}\\\hline
10	&Petersen graph	&\{3,2;1,1\}\\\hline
14	&Heawood graph	&\{3,2,2;1,1,3\}\\\hline
18	&Pappus graph	&\{3,2,2,1;1,1,2,3\}\\\hline
20  &Desargues graph	&\{3,2,2,1,1;1,1,2,2,3\}\\\hline
20	&Dodecahedral graph	&\{3,2,1,1,1;1,1,1,2,3\}\\\hline
\end{tabular}
\end{center}
\caption{A list of cubic distance-regular graphs we consider here}\label{table:cubic}
\end{table}

Here we deal with all of the cases with degree $\le 2$ (the $K_4$, Utility, or Petersen graphs) or with those already covered (Cube).  In these cases, the Markov chain problem is the classical solution of~\cite{Boyd.etal.09}.

When $G$ is the Utility Graph, then
\begin{equation*}
  \Phi(z;G) = \{(z+1) (2 z+1),(1-z) (1-2 z),(1-z) (z+1)\}
\end{equation*}
From this we see that $G$ is uniformly positive definite. 
When $G$ is the Petersen graph, then
\begin{equation*}
  \Phi(z;G) =\left\{1 + 3z + 6z^2,(1-z)^2,(1-z) (2 z+1)\right\}
\end{equation*}
This is, again, uniformly positive definite.

When $G$ is the Heawood graph, we have
\begin{align*}
  \Phi(z;G) &= \{(z+1) \left(4 z^2+2 z+1\right),(1-z) \left(4 z^2-2 z+1\right),(1-z) (z+1) \left(1-\sqrt{2} z\right),\\&(1-z) (z+1) \left(\sqrt{2} z+1\right)\}
\end{align*}
From this we see that $G$ is not uniformly positive definite. In fact, the matrix is no longer positive definite for $z > 1/\sqrt2$.

The optimal Markov chain problem is interesting, giving the following optimal choices for $d'=1,2,3$.  Recall the convention:  each row corresponds to a choice of $d'$, the first number is $\lmax$, the next vector is $\mu$, and the final vector is $\mu/n$.
\begin{equation*}
\left(
\begin{array}{ccc}
 \dfrac{1}{79} \left(12 \sqrt{2}+29\right) & \left\{1+\dfrac{6}{\sqrt{2}-9},-\dfrac{6}{\sqrt{2}-9}\right\} & \left\{1+\dfrac{6}{\sqrt{2}-9},-\dfrac{2}{\sqrt{2}-9}\right\} \\
 1-\dfrac{6}{\sqrt{2}+6} & \left\{\dfrac{\left(3 \sqrt{2}+16\right)}{238},\dfrac{3}{\sqrt{2}+6},\dfrac{3 \left(3 \sqrt{2}+16\right)}{119}\right\} & \left\{\dfrac{\left(3 \sqrt{2}+16\right)}{238},\dfrac{1}{\sqrt{2}+6},\dfrac{\left(3 \sqrt{2}+16\right)}{238} \right\} \\
 0 & \left\{\dfrac{1}{14},\dfrac{3}{14},\dfrac{3}{7},\dfrac{2}{7}\right\} & \left\{\dfrac{1}{14},\dfrac{1}{14},\dfrac{1}{14},\dfrac{1}{14}\right\} \\
\end{array}
\right) 
\end{equation*}

When $G$ is the Pappus graph, we have
\begin{align*}
  \Phi(z;G) &= \{(z+1) \left(2 z^3+4 z^2+2 z+1\right),(z-1) \left(2 z^3-4 z^2+2 z-1\right), \\&(1-z) (z+1) \left(z^2-\sqrt{3} z+1\right),(1-z) (z+1) \left(z^2+\sqrt{3} z+1\right),(1-z) (z+1) \left(1-2 z^2\right)\}
\end{align*}
From this we see that again $G$ is not uniformly positive definite. Interestingly, like the Heawood graph, it loses positivity for $z > 1/\sqrt2$.  The solutions for the optimal Markov chain problem are
\begin{equation*}
  \left(
\begin{array}{ccc}
 \dfrac{1}{13} \left(2 \sqrt{3}+5\right) & \left\{1+\dfrac{6}{\sqrt{3}-9},-\dfrac{6}{\sqrt{3}-9}\right\} & \left\{1+\dfrac{6}{\sqrt{3}-9},-\dfrac{2}{\sqrt{3}-9}\right\} \\
 1-\dfrac{8}{\sqrt{3}+9} & \left\{\dfrac{1}{39} \left(5-2 \sqrt{3}\right),\dfrac{4}{\sqrt{3}+9},\dfrac{4}{39} \left(\sqrt{3}+4\right)\right\} & \left\{\dfrac{1}{39} \left(5-2 \sqrt{3}\right),\dfrac{4}{3 \left(\sqrt{3}+9\right)},\dfrac{2}{117} \left(\sqrt{3}+4\right)\right\} \\
 \dfrac{1}{11} & \left\{\dfrac{1}{11},\dfrac{2}{11},\dfrac{4}{11},\dfrac{4}{11}\right\} & \left\{\dfrac{1}{11},\dfrac{2}{33},\dfrac{2}{33},\dfrac{2}{33}\right\} \\\\
 0 & \left\{\dfrac{1}{18},\dfrac{1}{6},\dfrac{1}{3},\dfrac{1}{3},\dfrac{1}{9}\right\} & \left\{\dfrac{1}{18},\dfrac{1}{18},\dfrac{1}{18},\dfrac{1}{18},\dfrac{1}{18}\right\} \\
\end{array}
\right)
\end{equation*}

Here we let $G$ be the Desargues graph and $H$ the Dodecahedral graph (these are the two cubic distance-regular graphs of order 20).  These are two cubic graphs with $20$ vertices.

\begin{align*}
\Phi(z;G) &= \{(1 + z) (1 + 2 z + 4 z^2 + 2 z^3 + z^4), (1 - z) (1 - 2 z + 
    4 z^2 - 2 z^3 + z^4), (1 - z)^2 (1 + z) (1 - z + 
    z^2),\\& (1 - z) (1 + z)^2 (1 + z + z^2), (1 - z)^3 (1 + 
    z)^2, (1 - z)^2 (1 + z)^3\}\\
    \Phi(z;H) &= \{(z+1) \left(z^4+2 z^3+4 z^2+2 z+1\right),(1-z) \left(z^4-\sqrt{5} z^3+z^3-\sqrt{5} z^2+3 z^2-\sqrt{5} z+z+1\right),\\&(1-z) \left(z^4+\sqrt{5} z^3+z^3+\sqrt{5} z^2+3 z^2+\sqrt{5} z+z+1\right),(1-z)^2 (z+1) \left(z^2-z+1\right),\\&(1-z)^2 (z+1)^3,(1-z) \left(z^4+z^3-2 z^2+z+1\right)
\end{align*}

For $G$ we have the optimal Markov chains
\begin{equation*}
 \left(
\begin{array}{ccc}
 \dfrac{5}{7} & \left\{\dfrac{1}{7},\dfrac{6}{7}\right\} & \left\{\dfrac{1}{7},\dfrac{2}{7}\right\} \\\\
 \dfrac{1}{3} & \left\{0,\dfrac{1}{3},\dfrac{2}{3}\right\} & \left\{0,\dfrac{1}{9},\dfrac{1}{9}\right\} \\\\
 \dfrac{1}{7} & \left\{\dfrac{3}{35},\dfrac{4}{35},\dfrac{12}{35},\dfrac{16}{35}\right\} & \left\{\dfrac{3}{35},\dfrac{4}{105},\dfrac{2}{35},\dfrac{8}{105}\right\} \\\\
 \dfrac{1}{19} & \left\{\dfrac{1}{19},\dfrac{3}{19},\dfrac{6}{19},\dfrac{6}{19},\dfrac{3}{19}\right\} & \left\{\dfrac{1}{19},\dfrac{1}{19},\dfrac{1}{19},\dfrac{1}{19},\dfrac{1}{19}\right\} \\\\
 0 & \left\{\dfrac{1}{20},\dfrac{3}{20},\dfrac{3}{10},\dfrac{3}{10},\dfrac{3}{20},\dfrac{1}{20}\right\} & \left\{\dfrac{1}{20},\dfrac{1}{20},\dfrac{1}{20},\dfrac{1}{20},\dfrac{1}{20},\dfrac{1}{20}\right\} \\
\end{array}
\right)
\end{equation*}
and for $H$ we have
\begin{equation*}
  \left(
\begin{array}{ccc}
 \dfrac{\sqrt{5}}{3} & \{0,1\} & \left\{0,\dfrac{1}{3}\right\} \\\\
 \dfrac{5}{13} & \left\{\dfrac{1}{13},0,\dfrac{12}{13}\right\} & \left\{\dfrac{1}{13},0,\dfrac{2}{13}\right\} \\\\
 0.137,&\{0.0549,0.184,0.298,0.463\},&\{0.0549,0.0614,0.0497,0.0771\} \\\\
 \dfrac{1}{19} & \left\{\dfrac{1}{19},\dfrac{3}{19},\dfrac{6}{19},\dfrac{6}{19},\dfrac{3}{19}\right\} & \left\{\dfrac{1}{19},\dfrac{1}{19},\dfrac{1}{19},\dfrac{1}{19},\dfrac{1}{19}\right\} \\\\
 0 & \left\{\dfrac{1}{20},\dfrac{3}{20},\dfrac{3}{10},\dfrac{3}{10},\dfrac{3}{20},\dfrac{1}{20}\right\} & \left\{\dfrac{1}{20},\dfrac{1}{20},\dfrac{1}{20},\dfrac{1}{20},\dfrac{1}{20},\dfrac{1}{20}\right\} \\
\end{array}
\right)
\end{equation*}

\section{Conclusions}\label{sec:outtro}

We have studied the spectra of generalized distance matrices and obtained a few results.

One of the most important components of our analysis \cb{in the examples} was exploiting the fact that for distance-regular graphs, or Cartesian products thereof, the eigenvalues are linear in the $f_i$.  This useful property is not true for graphs in general: for a simple example, consider $P_4$, the path graph on four vertices.  This has generalized distance matrix
\begin{equation*}
  \mc{M}(f;P_4) = \left(
\begin{array}{cccc}
 f_0 & f_1 & f_2 & f_3 \\
 f_1 & f_0 & f_1 & f_2 \\
 f_2 & f_1 & f_0 & f_1 \\
 f_3 & f_2 & f_1 & f_0
\end{array}
\right).
\end{equation*}
If we set $f_0=0$ and $f_1 = 2f_3$ we can obtain a nice formula for the four eigenvalues:
\begin{equation*}
  \dfrac12\left(\pm 3f_3 \pm \sqrt{4 f_2^2+16 f_3 f_2+17 f_3^2}\right).
\end{equation*}
\cb{The eigenvalues are not linear in the $f_i$.  We can check that $[A_1,A_2]\neq0$, so that the assumptions of Theorem~\ref{thm:commute} do not hold.}  A natural question is to identify the exact set of graphs that have the property that these eigenvalues are linear in the $f_i$. \cb{One might be tempted to think that the assumptions of Theorem~\ref{thm:commute} are sharp for this question.  For example, if we consider a linear combination of matrices all of whose eigenvalues are simple, then linearity in the $f_i$ would require the eigenvectors of the $A_k$ to match up to reordering and this would lead to commutativity.  However, the adjacency spectrum can have eigenvalues with high multiplicity so it is not a apriori clear that commutativity would be strictly required.}

We have also laid out a few conjectures about how parameters of the most rapidly mixing Markov Chain for the specific cases of certain families behave, especially in Section~\ref{sec:examples} above.  These seem complicated but tractable, since the quantities can all be expressed by some combinatorial identities.  Again, the fact that the spectrum is linear in the matrix elements is crucial.  Related problems have been considered in~\cite{Markowsky.Koolen.10, Koolen.Markowsky.Park.13}, and the results of this paper might give insight there as well.

Finally, we have shown how the simplicity of the expressions for the eigenvalues as functions of the coefficients of the matrix allows us to understand even nonlinear problems such as the GCLV model.  This gives significant insight into a fully nonlinear problem to an unexpected degree; in particular one can determine the parameter ranges for the stability of such systems to a degree (e.g. the $\mu$-domain that would stabilize a given nonlinearity as in Section~\ref{sec:Hamming}) that is uncommon for most nonlinear problems.

%

\end{document}